\documentclass{amsart}

\usepackage{amssymb,latexsym,amscd,amsmath,mathtools,tikz-cd}
\usepackage{hyperref}
\numberwithin{equation}{section}

\newtheorem{theorem}{Theorem}[section]
\newtheorem{maintheorem}[theorem]{Main Theorem}
\newtheorem{proposition}[theorem]{Proposition}

\newtheorem{corollary}[theorem]{Corollary}

\newtheorem{lemma}[theorem]{Lemma}

\theoremstyle{definition}

\newtheorem{remark}[theorem]{Remark}
\newtheorem{example}[theorem]{Example}
\newtheorem{definition}[theorem]{Definition}

\def\ZZ{\mathbb{Z}}

\def\CC{\mathbb{C}}

\def\wnot{w_\mathrm{o}}

\def\bb{\mathfrak{b}}
\def\gg{\mathfrak{g}}
\def\nn{\mathfrak{n}}
\def\hh{\mathfrak{h}}
\def\nn{\mathfrak{n}}
\def\ii{\mathbf{i}}

\def\kk{\Bbbk}

\def\sl{\mathfrak{sl}}

\usepackage{mathrsfs}
\begin{document}


\medskip

\title
{Factorizable Module Algebras}

\author{Arkady Berenstein}
\address{\noindent Department of Mathematics, University of Oregon,
Eugene, OR 97403, USA} \email{arkadiy@uoregon.edu}

\author{Karl Schmidt}
\address{\noindent Department of Mathematics, University of Oregon,
Eugene, OR 97403, USA} \email{karls@uoregon.edu}

\begin{abstract} The aim of this paper is to introduce and study a large class of $\gg$-module algebras which we call \emph{factorizable} by generalizing the Gauss factorization of square or rectangular matrices. This class includes coordinate algebras of corresponding reductive groups $G$, their parabolic subgroups, basic affine spaces and many others. It turns out that products of factorizable algebras are also factorizable and it is easy to create a factorizable algebra out of virtually any $\gg$-module algebra. We also have quantum versions of all these constructions in the category of $U_q(\gg)$-module algebras. Quite surprisingly, our quantum factorizable algebras are naturally acted on by the quantized enveloping algebra $U_q(\gg^*)$ of the dual Lie bialgebra $\gg^*$ of $\gg$. 
\end{abstract}

\maketitle

\tableofcontents

\section{Introduction}
\label{sect:Introduction}

The aim of this paper is to introduce and study a  class of $\gg$-module algebras which we call \emph{factorizable} by generalizing the Gauss factorization of square or rectangular matrices. 

More precisely, let $\gg$ be a complex semisimple Lie algebra. A commutative $\CC$-algebra $A$ which is also a $\gg$-module is called a $\gg$-module algebra if $\gg$ acts on $A$ by derivations. Given such a commutative $\gg$-module algebra $A$, we say that $A$ is \emph{factorizable} over a $\gg$-equivariant subalgebra $A_0$ if the restriction of the multiplication map of $A$, $\mu:A^+\otimes A_0\to A$, is an isomorphism of vector spaces. Here $A^+$ stands for the subalgebra of highest weight vectors in $A$, i.e., the kernel of the action of the maximal nilpotent subalgebra $\nn_+\subset \gg$. See Section \ref{sect:classical} for the precise definitions.

Factorizable algebras abound in ``nature" with $A_0=\CC[U]$, where $U$ is the maximal unipotent subgroup of the corresponding Lie group $G$, which is a $\gg$-module algebra isomorphic to  the graded dual of the Verma module $M_0$, as a $\gg$-module. Our first result shows that the natural objects in the representation theory of $\gg$ are factorizable.

\begin{theorem}
\label{thm:G/U}
	Up to a localization, the algebra $A=\CC[G/U]^{\otimes n}$ is factorizable over $A_0=\CC[U]$ for any $n\ge 1$.
\end{theorem}

\begin{remark}
\label{rem:classical Gauss}
	If we replace $\CC[G/U]$ with $A_1=\CC[x_1,\ldots,x_m]$ and let $\gg=\sl_m(\CC)$, then after localization by leading principal minors, $A_1^{\otimes n}=\CC[Mat_{m,n}]$ ($n\ge m-1$) is factorizable over $\CC[U_m^-]$ where $U_m^-$ is the group of all lower unitriangular matrices in $SL_m(\CC)$. This recovers the Gauss factorization of $m\times n$ matrices. See Example \ref{ex:mat} for details in a particular case.
\end{remark}

\begin{remark} Our factorizations are rather different from the well-known ones such as $U(\gg)=U(\gg_-)\otimes \nolinebreak U(\gg_+)$ prescribed by the Poincar\'e-Birkhoff-Witt theorem for any decomposition of $\gg$ into a direct sum of Lie subalgebras $\gg_-$ and $\gg_+$ or the Kostant harmonic decompositions of $S(\gg)$ into the $\gg$-invariants and $\gg$-harmonic elements.

\end{remark}

It turns out that factorizability of module algebras is not difficult to establish and reproduce.

\begin{theorem}
\label{thm:U+}

For any complex semisimple Lie algebra $\gg$, we have:

	(a) (Theorems \ref{thm:b-} \& \ref{thm:g}) Let $A$ be a $\gg$-module algebra containing a $\gg$-module subalgebra isomorphic to $\CC[U]$ and let $A^+$ denote the subalgebra of all highest weight vectors in $A$.	Then the algebra $A'=A^+\otimes \CC[U]$ has a natural $\gg$-module algebra structure such that $A'$ is factorizable over $\CC[U]$.
	
	
	(b) The assignments $A\mapsto A'$ define a functor $R$ from the category $\mathcal{A}_\gg$ of $\gg$-module algebras containing a $\gg$-module subalgebra isomorphic to $\CC[U]$ to the category ${\mathcal C}_\gg$ of $\gg$-module algebras which are factorizable over $\CC[U]$, a full subcategory of $\mathcal{A}_\gg$.
	
\end{theorem}

We can think of $R$ as a ``remembering" functor because it is right adjoint to the ``forgetful" functor $F:{\mathcal C}_\gg\to \mathcal{A}_\gg$. Clearly, the composition $R\circ F$ is isomorphic to the identity functor on $\mathcal{C}_\gg$. Now $\mathcal{A}_\gg$ has a natural tensor multiplication such that  for $A,B\in \mathcal{A}_\gg$, the tensor product $A\otimes B$ is a $\gg$-module algebra, plus the embedding $\CC[U]=1\otimes \CC[U]\subset A\otimes B$.

\begin{proposition} (Proposition \ref{prop:tensor}) $\mathcal{C}_\gg$ is closed under the natural tensor multiplication on $\mathcal{A}_\gg$.

%
	
\end{proposition}

However, neither category is monoidal because each lacks a unit object. It is possible to fix the issue in both categories by tensoring over $\CC[U]$ rather than over $\CC$. See the discussion in Remark \ref{rmk:unitissue}, where we make this more concrete.

It turns out that we can build factorizable algebras over $\CC[U]$ out of $\bb_-$-module algebras, where $\bb_-$ is the lower Borel subalgebra of $\gg$. In fact, all factorizable algebras can be obtained this way.

\begin{maintheorem}(Theorems \ref{thm:b-} \& \ref{thm:equiv}) 
\label{thm:construction} For any semisimple Lie algebra $\gg$, the assignments $A\mapsto A^+$ defines a functor $P$ from $\mathcal{A}_\gg$
to the category 
$\bb_- -{\bf ModAlg}$ of $\bb_-$-module algebras. Moreover, the composition $P\circ F$ is an equivalence of categories ${\mathcal C}_\gg\widetilde \to \bb_- -{\bf ModAlg}$.
 
%
\end{maintheorem}

\begin{remark} 
\label{rem:forgetmenot}
Informally speaking, the theorem asserts that the ``forgetful" highest weight vector functor $A\mapsto A^+$ from 
$\mathcal{A}_\gg$ to $Alg_\CC$, in fact, ``remembers almost everything."

\end{remark}


The  functor $P$ from Theorem \ref{thm:construction} is highly nontrivial: it involves a quite mysterious $\bb_-$ action on $A$ such that $A^+$ is a $\bb_-$-invariant subalgebra 
(see Section \ref{sect:classical} for details). 
Namely, the action of the Cartan subalgebra  of $\bb_-$ is inherited from that of $\gg$, but the action of the Chevalley generators $f_i$ of $\nn_-\subset \bb_-$ is given by the formula 
\begin{equation}
\label{eq:fihat}
f_i\triangleright a = f_i(a)-h_i(a)x_i
\end{equation}
for $a\in A$, where $h_i$ is the $i$-th Cartan generator and $x_i$ is the $i$-th ``near-diagonal element" in $\CC[U]$ ($\subset A$).

It is not difficult to show that $f_i\triangleright A^+\subset A^+$, but it is much harder to prove that the operators $f_i\triangleright$ satisfy the Serre relations (Theorem \ref{thm:b-}).

It turns out that all of the above results, including the mysterious Serre relations, can be quantized as well. Namely, we replace $\gg$ with its quantized enveloping algebra $U_q(\gg)$ and proceed as follows. Given a $U_q(\gg)$-module algebra $A_0$, we say that a $U_q(\gg)$-module algebra $A$ is \emph{factorizable} over a $U_q(\gg)$-equivariant subalgbra $A_0$ if the restriction of the multiplication map of $A$, $\mu:A^+\otimes A_0\to A$, is an isomorphism of vector spaces. By quantizing our default choice of $A_0$ in the classical case, we will now focus on $A_0=\CC_q[U]$, the quantized coordinate algebra of $U$, which is isomorphic to $U_q(\nn_+)$. See Section \ref{sect:quantum} for the precise definitions.


\begin{theorem}
\label{thm:qG/U}
	Up to a localization, the braided $n$-fold tensor power of $\CC_q[G/U]$, $A=\CC_q[G/U]^{\underline{\otimes} n}$, is factorizable over $A_0=\CC_q[U]$ for any $n\ge 1$.
\end{theorem}

\begin{remark} Similarly to the classical case (Remark \ref{rem:classical Gauss}), if we replace $\CC_q[G/U]$ with $A_1=\CC_q[x_1,\ldots,x_m]$, the algebra of $q$-polynomials and let $\gg=\sl_m(\CC)$, then the braided $n$-fold tensor power $A_1^{\underline{\otimes} n}=\CC_q[Mat_{m,n}]$ ($n\ge m-1$) is factorizable over $\CC_q[U_m]$, after localization by leading principal quantum minors. This recovers the Gauss factorization of quantum $m\times n$ matrices (see, e.g., \cite{dks}). See Example \ref{ex:qmat} for details in a particular case. 
\end{remark}

As in the classical case, factorizability of module algebras in the quantum case is also easy to establish and reproduce. 

\begin{theorem}
\label{thm:Uq+}

For any complex semisimple Lie algebra $\gg$, we have:

(a) (Theorems \ref{thm:qgd} \& \ref{thm:qg}) Let $A$ be a $U_q(\gg)$-module algebra containing a $U_q(\gg)$-module subalgebra isomorphic to $\CC_q[U]$ and let $A^+$ denote the subalgebra of all highest weight vectors in $A$. Then the vector space $A'=A^+\otimes \CC_q[U]$ has the structure of a $U_q(\gg)$-module algebra and is factorizable over $\CC_q[U]$.


(b) The assignments $A\mapsto A'$ define a functor $R_q$ from the category ${\mathcal{A}}_{\gg}^q$ of $U_q(\gg)$-module algebras containing a $U_q(\gg)$-module subalgebra isomorphic to $\CC_q[U]$ to the category ${\mathcal{C}}_{\gg}^q$ of $U_q(\gg)$-module algebras which are factorizable over $\CC_q[U]$, a full subcategory of ${\mathcal{A}}_{\gg}^q$.
\end{theorem}

We can think of $R_q$ as a ``remembering" functor because it is right adjoint to the ``forgetful" functor $F_q:{\mathcal{C}}_{\gg}^q\to {\mathcal{A}}_{\gg}^q$. Clearly, the composition $R_q\circ F_q$ is the identity functor on $\mathcal{C}_\gg^q$. In order to tensor multiply objects of these categories, we need to ``trim" it a bit. Namely, we consider the full subcategory $\underline{\mathcal{A}}_\gg^q$ consisting of \emph{weight} module algebras in $\mathcal{A}_\gg^q$ satisfying some additional mild conditions (see Section \ref{sect:quantum}). It turns out that $\underline{\mathcal{A}}_\gg^q$ has a natural \emph{braided} tensor product which we denote by $\underline{\otimes}$ (see, e.g. \cite{majid} and Section \ref{sect:quantum} below). Similar to the classical case, for $A, B\in \underline{\mathcal{A}}_{\gg}^{q}$, $A\underline{\otimes}B$ is naturally in $\underline{\mathcal{A}}_\gg^q$ with an embedding $\CC_q[U]=1\underline{\otimes} \CC_q[U]\subset A\underline{\otimes}B$. As in the classical case, this natural multiplication of course lacks a unit object.

Since $\mathcal{C}_\gg^q$ is a full subcategory of $\mathcal{A}_\gg^q$, we can define $\underline{\mathcal{C}}_\gg^q$ as the intersection of $\mathcal{C}_\gg^q$ and $\underline{\mathcal{A}}_\gg^q$.

\begin{proposition} (Proposition \ref{prop:qtensor}) $\underline{\mathcal{C}}_\gg^q$ is closed under the braided tensor multiplication in $\underline{\mathcal{A}}_{\gg}^{q}$.

%
\end{proposition}


That is, the category $\underline{\mathcal{C}}_\gg^q$ of factorizable $U_q(\gg)$-weight module algebras is ``almost" monoidal but it lacks a unit object as in the classical case.

Similar to the classical case (Theorem \ref{thm:construction}), we can build factorizable module algebras over $\CC_q[U]$ out of some module algebras. However, unlike the expected $U_q(\bb_-)$-module algebras, we will deal with $U_q(\gg^*)$-module algebras, where $\gg^*$ is the dual Lie bialgebra of $\gg$ and all factorizable algebras are obtained this way.

\begin{maintheorem}
\label{thm:constructionq}
(Theorems \ref{thm:qgd} \& \ref{thm:qequiv}) For any semisimple Lie algebra $\gg$, the assignments $A\mapsto A^+$ defines a functor $P_q$ from $\mathcal{A}_{\gg}^q$ to the category $U_q(\gg^*)-{\bf ModAlg}$ of $U_q(\gg^*)$-module algebras. Moreover, the composition $P_q\circ F_q$ is an equivalence of categories $\mathcal{C}_{\gg}^q\widetilde{\to}U_q(\gg^*)-{\bf ModAlg}$.
	
\end{maintheorem}

\begin{remark} 
\label{rem:forgetmenotq}
In the spirit of Remark \ref{rem:forgetmenot}, the theorem asserts that the assignment $A\mapsto A^+$ is the forgetful functor  which ``remembers almost everything."
\end{remark}

\begin{remark} We firmly believe that the emergence of the Lie bialgebra $\gg^*$ here is not a mere coincidence, but rather a manifestation of the ``semiclassical story" behind the quantum one. We plan to investigate it in a separate publication, when all relevant objects are Poisson algebras with a compatible action of the Poisson-Lie group $G$, where its Poisson-Lie dual $G^*$ emerges naturally.
\end{remark}


The functor $P_q$ from Theorem \ref{thm:constructionq} is highly nontrivial: it involves a quite mysterious $U_q(\gg^*)$ action on $A$ such that $A^+$ is a $U_q(\gg^*)$-invariant subalgebra (see Section \ref{sect:quantum} for details). Namely, the Cartan subalgebra action of $U_q(\gg^*)$ is inherited from that of $U_q(\gg)$, but the action of the generators $F_{i,1}$ and $F_{i,2}$ of $U_q(\gg^*)$ is given by the formulas
\begin{equation}
\label{eq:fihatq}
F_{i,1}\triangleright a = F_i(a)-\frac{x_ia-K_i^{-1}(a)x_i}{q_i-q_i^{-1}}, \quad F_{i,2}\triangleright a = \frac{x_iK_i^{-1}(a)-ax_i}{q_i-q_i^{-1}}
\end{equation}
for $a\in A^+$, where $K_i$ is the $i$-th Cartan generator of $U_q(\gg)$, $q_i=q^{d_i}$, and $x_i$ is the $i$-th generator of $\CC_q[U]\subset A$.

As in the classical case, the most non-trivial part of the proof of the main theorem is to show that  operators $F_{i,1}\rhd$ and $F_{i,2}\rhd$ given by \eqref{eq:fihatq} satisfy the quantum Serre relations, which we establish in Theorem \ref{thm:qgd}. Also, while not difficult, it is still surprising that in \eqref{eq:fihatq}, all operators $F_{i,1}\rhd$ commute with all  $F_{i,2}\rhd$.


\section{Definitions, Notation, and Results}
\label{sect:Definitions, Notation, and Results}

In this section, we will recall and introduce the relevant definitions and notation necessary to present our main results, which will also be included. We begin by defining the main object of study in this paper: module algebras.

\begin{definition}
	Let $\kk$ be an arbitrary field and $H$ a $\kk$-bialgebra. A $\kk$-algebra $A$ is called an \emph{$H$-module algebra} if it is an $H$-module, multiplication $(-)\cdot(-):A\otimes_\kk A\to A$ is a homomorphism of $H$-modules, and $h(1)=\epsilon(h)$ for $h\in H$, where $\epsilon$ is the counit of $H$. That is to say, if we denote in sumless Sweedler notation $\Delta(h)=h_{(1)}\otimes h_{(2)}$ for $h\in H$, then for $a,b\in A$, $h(a\cdot b)=h_{(1)}(a)\cdot h_{(2)}(b)$ and $h(1)=\epsilon(h)$. A homomorphism of $H$-module algebras is a homomorphism of $H$-modules and algebras. We will denote the category of $H$-module algebras by $H$-{\bf ModAlg}. If $\gg$ is a Lie algebra over $\kk$, then we will shorten ``$U(\gg)$-module algebra" and ``$U(\gg)-{\bf ModAlg}$" to ``$\gg$-module algebra" and ``$\gg$-{\bf ModAlg}", respectively.
\end{definition}

	Let $\gg$ be a semisimple complex Lie algebra, with $I\times I$ symmetrizable Cartan matrix $C=(c_{i,j})$ and fixed choice of symmetrizers, $(d_i)_{i\in I}$, for $C$, i.e. $d_ic_{i,j}=d_jc_{j,i}$ for $i,j\in I$. Then $\gg$ has a triangular decomposition $\gg=\nn_-\oplus\hh\oplus\nn_+$. Here $\hh$ is a Cartan subalgebra with $\dim \hh=|I|$ and its dual $\hh^*$ has basis $\{\alpha_i\ |\ i\in I\}$, the simple roots of the associated root system. Let $(\cdot,\cdot)$ be the symmetric bilinear form on $\hh^*$ satisfying $(\alpha_i,\alpha_j)=d_ic_{i,j}$. As usual, we set $\alpha_i^\vee:=\frac{2\alpha_i}{(\alpha_i,\alpha_i)}$. By $\Lambda=\{\lambda\in \hh^*\ |\ (\alpha_i^\vee,\lambda)\in \ZZ\ \forall i\in I\}$, we denote the set of integral weights. Denote by $\omega_i\in \Lambda$ is the $i$-th fundamental weight, which satisfies $(\alpha_i^\vee,\omega_j)=\delta_{i,j}$ for $i,j\in I$. Let $W$ be the Weyl group of $\gg$, generated by the simple reflections $\{s_i\ |\ i\in I\}$, and with longest element $\wnot$. Given $w\in W$, $R(w)$ is the set of all ${\bf i}=(i_1,i_2,\ldots,i_m)\in I^m$ such that $s_{i_1}s_{i_2}\cdots s_{i_m}$ is a reduced expression for $w$.

\subsection{Quantum Factorization}
\label{sect:quantum}
Throughout this section, all tensor products will be taken over $\CC(q)$ unless otherwise specified and written $-\otimes-$ rather than $-\otimes_{\CC(q)}-$.
	
	 Recall that $U_q(\gg)$ is the quantum enveloping algebra of $\gg$, generated by elements $\{K_i^{\pm 1},E_i,F_i\ |\ i\in I\}$ subject to the relations
	\begin{center}
		$\begin{array}{c  c}
			K_iK_j=K_jK_i; &K_iK_i^{-1}=K_i^{-1}K_i=1;\\\\
			K_iE_jK_i^{-1}=q_i^{c_{i,j}}E_j;& K_iF_jK_i^{-1}=q_i^{-c_{i,j}}F_j;
		\end{array}$
		\[E_iF_j-F_jE_i=\delta_{i,j}\frac{K_i-K_i^{-1}}{q_i-q_i^{-1}};\]
		\[\sum_{k=0}^{1-c_{i,j}}(-1)^k E_i^{(k)}E_jE_i^{(1-c_{i,j}-k)}=0\text{ if }i\ne j;\]
		\[\sum_{k=0}^{1-c_{i,j}}(-1)^k F_i^{(k)}F_jF_i^{(1-c_{i,j}-k)}=0\text{ if }i\ne j;\]
	\end{center}
where $y_i^{(n)}=\frac{1}{(n)_{q_i}!}y_i^n$, $(n)_{q_i}!=(1)_{q_i}(2)_{q_i}\cdots (n)_{q_i}$, $(m)_{q_i}=\dfrac{q_i^m-q_i^{-m}}{q_i-q_i^{-1}}$, and $q_i=q^{d_i}$.
$U_q(\gg)$ is a Hopf algebra with comultiplication $\Delta$, counit $\epsilon$, and antipode $S$ given on generators by
	
	\begin{center}
	$\begin{array}{c c c}
		\Delta(K_i^{\pm1})=K_i^{\pm1}\otimes K_i^{\pm1}; & \epsilon(K_i^{\pm1})=1 & S(K_i^{\pm1})=K_i^{\mp1};\\\\
		\Delta(E_i)=E_i\otimes K_i+1\otimes E_i; & \epsilon(E_i)=0; & S(E_i)=-E_iK_i^{-1};\\\\
		\Delta(F_i)=F_i\otimes 1+K_i^{-1}\otimes F_i; & \epsilon(F_i)=0; & S(F_i)=-K_iF_i.\\\\
	\end{array}$
	\end{center}

Now, $\mathcal{K}$ is the Hopf subalgebra of $U_q(\gg)$ generated by all $K_i$ as a $\CC(q)$-algebra, while $U_q(\bb_+)$ (respectively $U_q(\bb_-)$) is the Hopf subalgebra of $U_q(\gg)$ generated by all $K_i^{\pm1}$ and $E_i$ (respectively $K_i^{\pm 1}$ and $F_i$). We will assume henceforth that for any $i\in I$, the action of $E_i$ on any $U_q(\bb_+)$-modules is \emph{locally nilpotent}. In other words, if $M$ is a $U_q(\bb_+)$-module, we will assume that for each $x\in M$ and $i\in I$, there exists some $n\ge 0$ such that $E_i^n(x)=0$. Note that every $U_q(\gg)$-module is a $U_q(\bb_+)$-module, so we are assuming these are ``bounded above". We do not assume the same for the action of $F_i$. For a $U_q(\bb_+)$-module $M$, we designate $M^+:=\{m\in M\ |\ E_i(m)=0\ \forall i\in I\}$, the set of \emph{highest weight vectors}. If $A$ is a $U_q(\bb_+)$-module algebra, then $A^+$ is a $U_q(\bb_+)$-module subalgebra.

Suppose $M$ is a $U_q(\bb_+)$-module. For each $i\in I$ and $x\in M\setminus\{0\}$, set $\ell_i(x)=\max\{\ell\in\ZZ_{\ge0}\ |\ E_i^\ell(x)\ne 0\}$ and $E_i^{(top)}(x)=E_i^{(\ell_i(x))}(x)$. 
Given ${\bf i}\in I^m$ for some $m\ge 0$ and $x\in M\setminus\{0\}$, we also use the shorthand 
	\[E_{\bf i}^{(top)}(x)=E_{i_m}^{(top)}E_{i_{m-1}}^{(top)}\cdots E_{i_1}^{(top)}(x)\]
and define $\nu_{\bf i}:M\setminus \{0\}\to \ZZ_{\ge 0}^m$, $x\mapsto(a_1,a_2,\ldots,a_m)$ by the following:
	\[a_k=\ell_{i_k}(E_{i_{k-1}}^{(top)}E_{i_{k-2}}^{(top)}\cdots E_{i_1}^{(top)}(x)).\]
	
Lastly, for ${\bf j}=(j_1,\ldots,j_m)\in \ZZ_{\ge 0}^m$, we set $E_{\bf i}^{(\bf j)}:=E_{i_m}^{(j_m)}E_{i_{m-1}}^{(j_{m-1})}\cdots E_{i_1}^{(j_1)}.$

\begin{definition}
Let $A$ be a $U_q(\bb_+)$-module algebra, $w\in W$, and ${\bf i}\in R(w)$. If $E_{\bf i}^{ (top)}(x)\in A^+$ for all $x\in A\setminus\{0\}$, then we say $A$ is \emph{{\bf i}-adapted}. We say a basis $\mathcal{B}$ for $A$ is an \emph{{\bf i}-adapted basis} if
	\begin{enumerate}
		\item $E_{\bf i}^{(top)}(b)=1$ for all $b\in \mathcal{B}$.
		\item The restriction of $\nu_{\bf i}$ to $\mathcal{B}$ is an injective map $\mathcal{B}\hookrightarrow \ZZ_{\ge 0}^m$, where $m$ is the length of $w$.
	\end{enumerate}

If there exists any $w\in W$ and ${\bf i}\in R(w)$ so that $A$ is {\bf i}-adapted, then we say more generally that $A$ is \emph{adapted}. 
\end{definition}


\begin{remark}
	Our notion of an {\bf i}-adapted $U_q(\bb_+)$-module algebra is different than P. Caldero's notion of adapted algebra in \cite{caldero}, though they do have some examples in common. On the other hand, our notion of {\bf i}-adapted basis is stronger than the similar notion of an adapted basis for $(A,\nu_{\bf i})$ as in \cite{kavehmanon}.
\end{remark} 

It turns out that if $A_0$ possesses an ${\bf i}$-adapted basis for some ${\bf i}\in R(w)$ and is a ``large enough" $U_q(\bb_+)$-module subalgebra of $A$, then $A$ is factorizable over $A_0$. The following theorem makes this precise.

\begin{theorem}
\label{thm:qfactoring}
Let $A$ be a $U_q(\bb_+)$-module algebra. Suppose $A_0$ is a $U_q(\bb_+)$-module subalgebra of $A$ possessing an {\bf i}-adapted basis $\mathcal{B}$ for some reduced $\ii$. Then:
\begin{enumerate}
	\item The restriction $\mu:A^+\otimes A_0\to A$ of the multiplication in $A$ is an injective homomorphism of $U_q(\bb_+)$-modules.
	\item The map $\mu$ is an isomorphism if and only if $A$ is {\bf i}-adapted and $\nu_{\bf i}(A\setminus\{0\})=\nu_{\bf i}(A_0\setminus\{0\})$.
\end{enumerate}
\end{theorem}

We will prove Theorem \ref{thm:qfactoring} in Section \ref{pf:qfactoring}. Theorem \ref{thm:qfactoring} demonstrates a close relationship between being {\bf i}-adapted and being factorizable over a $U_q(\bb_+)$-module subalgebra possessing an {\bf i}-adapted basis. The following theorem explores this relationship from a different angle.

\begin{theorem}
\label{thm:qclassify}
	Let $A$ be an {\bf i}-adapted $U_q(\bb_+)$-module algebra for some reduced $\ii$ and suppose $A_0$ is a $U_q(\bb_+)$-module subalgebra of $A$. Then $\mu:A^+\otimes A_0\to A$ as in Theorem \ref{thm:qfactoring} is an isomorphism of $U_q(\bb_+)$-modules if and only if $A_0$ possesses an {\bf i}-adapted basis and $\nu_{\bf i}(A_0\setminus\{0\})=\nu_{\bf i}(A\setminus\{0\})$.
\end{theorem}

Theorem \ref{thm:qclassify} is proved in Section \ref{pf:qclassify}. We now restrict our focus to a specific $U_q(\gg)$-module algebra, namely $\CC_q[U]$. As a $\CC(q)$-algebra, $\CC_q[U]$ is generated by the set $\{x_i\ |\ i\in I\}$, subject to the quantum Serre relations:
	\[\sum_{k=0}^{1-c_{i,j}}(-1)^k x_i^{(k)}x_jx_i^{(1-c_{i,j}-k)}=0\text{ if }i\ne j.\]
The $U_q(\gg)$-module structure on $\CC_q[U]$ is summarized in the following equations:
\begin{align*}
	K_i^{\pm1}(x_j)&=q_i^{\mp c_{i,j}}x_j \quad \text{for all } i,j\in I\\
	E_i(x_j)&=\delta_{i,j}\quad \text{for all } i,j\in I\\
	F_i(x)&=\frac{x_ix-K_i^{-1}(x)x_i}{q_i-q_i^{-1}}\quad \text{for all } i\in I \text{ and } x\in \CC_q[U].
\end{align*}

Of course, the actions of $K_i^{\pm1}$ and $E_i$ must be extended to all of $\CC_q[U]$ by the rules
	\begin{align*}
		K_i^{\pm1}(xx')&=K_i^{\pm1}(x)K_i^{\pm1}(x')\quad \text{for all } i\in I \text{ and } x,x'\in \CC_q[U]\\
		E_i(xx')&=E_i(x)K_i(x')+xE_i(x')\quad \text{for all } i\in I,\text{ and } x,x'\in \CC_q[U].
	\end{align*}
	
The first author and A. Zelevinsky observed in \cite[Proposition 3.5]{berensteinzelevinsky} that $\CC_q[U]$ possesses a basis $\mathcal{B}^{dual}$ such that, for ${\bf i}\in R(\wnot)$, the restriction of $\nu_{\bf i}$ to $\mathcal{B}^{dual}$ is injective. Note that they use the notation $\mathcal{A}$ in place of $\CC_q[U]$ and view it only as a $U_q(\nn_+)$-module. As hinted by the notation, $\mathcal{B}^{dual}$ is the so-called \emph{dual canonical basis}. In Section \ref{pf:adapted}, we prove the following proposition.

\begin{proposition}
\label{prop:adapted}
	Given any ${\bf i}\in R(\wnot)$, $\mathcal{B}^{dual}$ is an ${\bf i}$-adapted basis for $\CC_q[U]$.
\end{proposition}

\begin{remark}
Based on the recent paper \cite{kimuraoya}, we expect that the dual canonical basis $\mathcal{B}^{dual}\cap U_q(w)$ in each quantum Schubert cell $U_q(w)$  is $\ii$-adapted for any reduced word $\ii$ for $w$.
\end{remark}

Combining Proposition \ref{prop:adapted} with Theorem \ref{thm:qfactoring}, we are led to the following corollary, though it does still require some proof.

\begin{corollary}
\label{cor:qhighest}
	Let $A$ be a $U_q(\gg)$-module algebra containing $\CC_q[U]$ as a $U_q(\gg)$-module subalgebra. If there exists a $U_q(\gg)$-module algebra $A'$ containing $A$ as a $U_q(\gg)$-module subalgebra, such that $A'$ is generated by $(A')^+$ as a $U_q(\gg)$-module algebra, then $\mu:A^+\otimes \CC_q[U]\to A$ as in Theorem \ref{thm:qfactoring} is an isomorphism of $U_q(\bb_+)$-modules.
\end{corollary}

Corollary \ref{cor:qhighest} will be proved in Section \ref{pf:qhighest} and provides us with the means to prove Theorem \ref{thm:qG/U}, which we do in Section \ref{pf:qG/U}. In the meantime, there are two families of quantities that arose in the proof of Corollary \ref{cor:qhighest}:
	\[x_ia-K_i(a)x_i\quad\text{and}\quad F_i(a)+x_i\frac{K_i^{-2}(a)-a}{q_i-q_i^{-1}}\]
where $i\in I$ and $a\in A'$. These quantities are equally valid to consider for $a\in A$, without the assumed presence of $A'$. If $a\in A^+$, then both of these quantities are also in $A^+$. It is therefore natural to ask what relations the families of operators $L_i-R_iK_i$ and $F_i+L_i\frac{K_i^{-2}-1}{q_i-q_i^{-1}}$ satisfy, where $L_i$ (respectively $R_i$) represents left (respectively right) multiplication by $x_i$. Or to put it another way, do these operators indicate the action of a known algebra which is somehow related to $U_q(\gg)$? We can answer in the affirmative. It can be proved that both of the families of operators observed in fact satisfy the quantum Serre relations and the two families ``almost" commute with each other. This resembles an action of the Hopf algebra $U_q(\gg^*)$, which we now define for the reader's convenience.

\begin{definition}
	As an algebra, $U_q(\gg^*)$ is generated by $\{K_i^{\pm1}, F_{i,1}, F_{i,2}\ |\ i\in I\}$ subject to the following relations for $i,j\in I$ and $k\in \{1,2\}$:
		\begin{center}
		$\begin{array}{c c}
			K_i^{\pm1}K_j^{\pm1}=K_j^{\pm1}K_i^{\pm1}; & K_iK_i^{-1}=K_i^{-1}K_i=1;\\[.2cm]
			F_{i,1}F_{j,2}=F_{j,2}F_{i,1}; & K_iF_{j,k}K_i^{-1}=q_i^{-c_{i,j}}F_{j,k};
		\end{array}$
		\[\sum_{\ell=0}^{1-c_{i,j}}(-1)^kF_{i,k}^{(\ell)}F_{j,k}F_{i,k}^{(1-c_{i,j}-\ell)}=0 \text{ if }i\ne j.\]
		\end{center}
	The comultiplication $\Delta$, counit $\epsilon$, and antipode $S$ are given on generators by
	
	\begin{center}
		$\begin{array}{c c c}
			\Delta(K_i^{\pm1})=K_i^{\pm1}\otimes K_i^{\pm1}; & \epsilon(K_i^{\pm1})=1; & S(K_i^{\pm1})=K_i^{\mp1};\\\\
			\Delta(F_{i,1})=F_{i,1}\otimes 1+K_i^{-1}\otimes F_{i,1}; & \epsilon(F_{i,1})=0; & S(F_{i,1})=-K_iF_{i,1};\\\\
			\Delta(F_{i,2})=F_{i,2}\otimes K_i^{-1}+1\otimes F_{i,2}; & \epsilon(F_{i,2})=0; & S(F_{i,2})=-F_{i,2}K_i.
		\end{array}$
	\end{center}
\end{definition}

After some tweaking and combining of our operators with the inherited Cartan action, we see that our operators really do indicate the presence of a $U_q(\gg^*)$-module algebra structure. The following theorem summarizes this and is proved in Section \ref{pf:qgd,qg}.

\begin{theorem}
\label{thm:qgd}
	Let $A$ be a $U_q(\gg)$-module algebra containing $\CC_q[U]$ as a $U_q(\gg)$-module subalgebra. Then $A$ is a $U_q(\gg^*)$-module algebra with action given by
	\[K_i^{\pm 1}\triangleright a = K_i^{\pm1}(a),\quad F_{i,1}\triangleright a = F_i(a)-\frac{x_ia-K_i^{-1}(a)x_i}{q_i-q_i^{-1}},\quad F_{i,2}\triangleright a = \frac{x_iK_i^{-1}(a)-ax_i}{q_i-q_i^{-1}}.\]
	In particular, the subalgebra $A^+$ is invariant under this action of $U_q(\gg^*)$ and is therefore a $U_q(\gg^*)$-module subalgebra.
\end{theorem}

Theorem \ref{thm:qgd} is in some sense a statement about the existence of a functor. To make this precise, we introduce a category whose objects bear properties similar to those found in Theorem \ref{thm:qfactoring}.

\begin{definition}
	Let $\mathcal{C}_\gg^q$ be the category whose objects consist of pairs $(A,\varphi_A)$, where 
	\begin{itemize}
		\item $A$ is an adapted $U_q(\gg)$-module algebra such that $\nu_{\bf i}(A\setminus \{0\})=\nu_{\bf i}(\CC_q[U]\setminus\{0\})$ for all ${\bf i}\in R(\wnot)$.
		\item $\varphi_A:\CC_q[U]\hookrightarrow A$ is an embedding of $U_q(\gg)$-module algebras.
	\end{itemize}
	A morphism $(A,\varphi_A)\to (B,\varphi_B)$ in $\mathcal{C}_\gg^q$ is a homomorphism of $U_q(\gg)$-module algebras $\psi:A\to B$ such that $\psi\circ\varphi_A=\varphi_B$.
\end{definition}

Given a homomorphism of $U_q(\gg)$-module algebras $\psi:A\to B$, it follows that $\psi(A^+)\subseteq B^+$, so $\psi|_{A^+}$ may be thought of as a map of $\CC(q)$-algebras $A^+\to B^+$. If $\psi$ is a morphism in $\mathcal{C}_\gg^q$, $(A,\varphi_A){\to} (B,\varphi_B)$, then actually $\psi|_{A^+}$ is a homomorphism of $U_q(\gg^*)$-module algebras. As a consequence, we have the following corollary.

\begin{corollary}
	There is a functor $(-)^+:\mathcal{C}_\gg^q\to U_q(\gg^*)-\textbf{ModAlg}$ (denoted $P_q\circ F_q$ in Section \ref{sect:Introduction}) which assigns to an object $(A,\varphi_A)$ of $\mathcal{C}_\gg^q$ its subalgebra of highest weight vectors $A^+$, equipped with the $U_q(\gg^*)$-module algebra structure of Theorem \ref{thm:qgd}. The functor $(-)^+$ is given on morphisms by restriction.
\end{corollary}

Theorem \ref{thm:qfactoring} strongly suggests that $(-)^+$ might actually be an equivalence of categories. In fact, this is the case, but in order to describe a quasi-inverse, we need the following theorem which describes a $U_q(\gg)$-module algebra structure on $A\otimes \CC_q[U]$ if $A$ is a $U_q(\gg^*)$-module algebra.

\begin{theorem}
\label{thm:qg}
	If $A$ is a $U_q(\gg^*)$-module algebra, then $A\otimes \CC_q[U]$ has the structure of a $U_q(\gg)$-module algebra determined by:
	\begin{align*}
	(1\otimes x_i)(a\otimes 1)&=K_i(a)\otimes x_i+(q_i-q_i^{-1})F_{i,2}K_i(a)\otimes 1\\
	K_i^{\pm1}\triangleright(a\otimes x)&=K_i^{\pm1}(a)\otimes K_i^{\pm1}(x)\\
	E_i\triangleright(a\otimes x)&=a\otimes E_i(x)\\
	F_i\triangleright(a\otimes x)&=(F_{i,1}(a)+F_{i,2}K_i(a))\otimes x+\frac{K_i(a)-K_i^{-1}(a)}{q_i-q_i^{-1}}\otimes x_ix+K_i^{-1}(a)\otimes F_i(x).
	\end{align*}
\end{theorem}

Theorem \ref{thm:qg} will be proved in Section \ref{pf:qgd,qg}. Since the action of each $E_i$ is completely described on the $\CC_q[U]$ factor and $\CC_q[U]$ is adapted, we have the following corollary.

\begin{corollary}
\label{cor:qgcor}
	There is a functor $(-)\otimes \CC_q[U]:U_q(\gg^*)-\textbf{ModAlg}\to \mathcal{C}_\gg^q$ which assigns to a $U_q(\gg^*)$-module algebra $A$, the pair $(A\otimes \CC_q[U], 1\otimes \text{id})$, where $A\otimes \CC_q[U]$ is given the $U_q(\gg)$-module structure of Theorem \ref{thm:qg}. The functor $(-)\otimes \CC_q[U]$ is given on morphisms by $\psi\mapsto \psi\otimes \text{id}$.
\end{corollary}

The following theorem says that $(-)\otimes \CC_q[U]$ is the promised quasi-inverse for $(-)^+$.

\begin{theorem}
\label{thm:qequiv}
	The functors $(-)^+:\mathcal{C}_\gg^q\to U_q(\gg^*)-{\bf ModAlg}$ and $(-)\otimes \CC_q[U]:U_q(\gg^*)-{\bf ModAlg}\to \mathcal{C}_\gg^q$ are quasi-inverses of each other and thus provide equivalences of categories.
\end{theorem}

Theorem \ref{thm:qequiv} is proved in section \ref{pf:qequiv}. Now, it is well-known that if $A$ and $B$ are $U_q(\gg)$-weight module algebras, then so is the braided tensor product $A\underline{\otimes}B$. Here $A\underline{\otimes}B$ has multiplication given by \[(a\otimes b)(a'\otimes b')=a\mathcal{R}_{(2)}(a')\otimes \mathcal{R}_{(1)}(b)b',\] where $\mathcal{R}$ is the universal $R$-matrix for $U_q(\gg)$ and we use sumless Sweedler notation $\mathcal{R}=\mathcal{R}_{(1)}\otimes \mathcal{R}_{(2)}$. The $R$-matrix is of the form 
\[\mathcal{R}=q^{\sum\limits_{i,j}d_i(C^{-1})_{i,j}H_i\otimes H_j}\left(\prod\limits_{\alpha>0}^{\leftarrow} e_{q_\alpha^{-2}}^{(q_{\alpha}-q_{\alpha}^{-1})E_\alpha\otimes F_\alpha}\right)\]
(see \cite[Section 3.3]{majid} for details). At this point, we may need to use the field $\CC(q^{1/d})$ where $d$ is the determinant of $C$ or else assume that $d_i(C^{-1})_{i,j}\in \ZZ$ for all $i,j\in I$, but this is a small matter which doesn't affect our overall approach. We define a subcategory of $\mathcal{C}_\gg^q$ on which the tensor product defined above will make sense.
 
\begin{definition}
	Let $\underline{\mathcal{C}}_\gg^q$ be the full subcategory of $\mathcal{C}_\gg^q$ whose objects consist of pairs $(A,\varphi_A)$, where $A$ is additionally assumed to be a $U_q(\gg)$-weight module algebra.
\end{definition}

The following proposition is then clear.

\begin{proposition}
\label{prop:qequiv}
	The functors $(-)^+$ and $(-)\otimes \CC_q[U]$ restrict to equivalences between $\underline{\mathcal{C}}_\gg^q$ and $U_q(\gg^*)-{\bf WModAlg}$, the category of $U_q(\gg^*)$-weight module algebras.
\end{proposition}

If $(A,\varphi_A)$ and $(B,\varphi_B)$ are objects of $\underline{\mathcal{C}}_\gg^q$, then we already saw that $A\underline{\otimes}B$ is also a $U_q(\gg)$-weight module algebra. Furthermore, it is obvious that $1\otimes \varphi_B$ and $\varphi_A\otimes 1$ are injections $\CC_q[U]\hookrightarrow A\underline{\otimes} B$. However, it is not immediately obvious that $A\underline{\otimes} B$ is adapted with $\nu_{\bf i}(A\underline{\otimes} B\setminus\{0\})=\nu_{\bf i}(\CC_q[U]\setminus\{0\})$ for all ${\bf i}\in R(\wnot)$. Nevertheless, this is the case, which the following proposition asserts.

\begin{proposition}
\label{prop:qtensor}
	If $(A,\varphi_A)$ and $(B,\varphi_B)$ are objects of $\underline{\mathcal{C}}_\gg^q$, then $(A\underline{\otimes}B, 1\otimes \varphi_B)$ and $(A\underline{\otimes}B, \varphi_A\otimes 1)$ are objects of $\underline{\mathcal{C}}_\gg^q$ as well.
\end{proposition}

Proposition \ref{prop:qtensor} is proved in Section \ref{pf:qtensor}. Proposition \ref{prop:qequiv} allows us to turn Proposition \ref{prop:qtensor} into a statement about $U_q(\gg^*)$-module algebras. We define two ``fusion" products on the category of $U_q(\gg^*)$-weight module algebras, namely the following:
	\begin{align*}
		A*B&:=((A\otimes \CC_q[U])\underline{\otimes} (B\otimes \CC_q[U]),1\otimes 1\otimes 1\otimes \text{id})^+\\
		A\star B&:=((A\otimes \CC_q[U])\underline{\otimes} (B\otimes \CC_q[U]),1\otimes \text{id}\otimes 1\otimes 1)^+.
	\end{align*}
	
These fusion products are associative, but not monoidal due to the easy observation that there is no unit object. The reader may be bothered that the objects $(A\underline{\otimes}B,1\otimes \varphi_B)$ and $(A\underline{\otimes}B,\varphi_A\otimes 1)$ are not (necessarily at least) isomorphic despite having equal underlying $U_q(\gg)$-module algebras. An attempt to force a common quotient leads to the discovery of an interesting $U_q(\gg^*)$-module algebra structure on $A\otimes B$ if $A$ and $B$ are $U_q(\gg^*)$-weight module algebras.

\begin{proposition}
\label{prop:qfusion}
Let $A$ and $B$ be $U_q(\gg^*)$-weight module algebras. Then the $\CC(q)$-vector space $A\otimes B$ has the structure of a $U_q(\gg^*)$-weight module algebra satisfying the following equations
	\begin{align*}
		(a\otimes b)(a'\otimes b')&=q^{(|a'|,|b|)}aa'\otimes bb'\\
		K_i^{\pm1}\rhd(a\otimes b)&=K_i^{\pm1}(a)\otimes K_i^{\pm1}(b)\\
		F_{i,1}\rhd(a\otimes b)&=K_i^{-1}(a)\otimes F_{i,1}(b)\\
		F_{i,2}\rhd(a\otimes b)&=F_{i,2}(a)\otimes K_i^{-1}(b).
	\end{align*}
for weight vectors $a,a'\in A$, $b,b'\in B$, and $i\in I$, where $|\cdot|$ indicates weight.
\end{proposition}

Proposition \ref{prop:qfusion} is proved in Section \ref{pf:qfusion} and induces a fusion product on $\underline{\mathcal{C}}_\gg^q$:
	\[(A,\varphi_A)\diamond (B,\varphi_B):=(((A,\varphi_A)^+\otimes (B,\varphi_B)^+)\otimes \CC_q[U],1\otimes 1\otimes \text{id}).\]
Just like for $*$ and $\star$, there is no unit object for $\diamond$, so it is not a monoidal tensor product.


\subsection{Classical Factorization}
\label{sect:classical}

Throughout this section, all tensor products will be taken over $\CC$ unless otherwise specified and written $-\otimes-$ rather than $-\otimes_\CC -$. We also assume henceforth that every algebra is commutative unless otherwise stated, with the exception of previously referenced algebras such as $\CC_q[U]$.

The semisimple complex Lie algebra $\gg$ with Cartan matrix $C=(c_{i,j})$ is generated by elements $\{e_i,f_i,h_i\ |\ i\in\nolinebreak I\}$ subject to the following relations:
	$$
		[h_i,h_j]=0,~[e_i,f_j]=\delta_{i,j}h_i,~\null[h_i,e_j]=c_{i,j}e_j,~[h_i,f_j]=c_{i,j}f_j,~
(\text{ad } e_i)^{1-c_{i,j}}(e_j)=(\text{ad } f_i)^{1-c_{i,j}}(f_j)=0,$$
where as usual $(\text{ad }x)(y)=[x,y]$ for $x,y\in \gg$. The universal enveloping algebra $U(\gg)$ of $\gg$ is a noncommutative Hopf algebra on the same generators and relations, where $[x,y]=xy-yx$ for $x,y\in U(\gg)$. The comultiplication of $U(\gg)$ is given on generators by 
	\[\Delta(x)=x\otimes 1+1\otimes x\]
for $x\in \{e_i,f_i,h_i\ |\ i\in I\}$.
	
We denote by $\nn_+$ (respectively $\bb_-$) the Lie subalgebra of $\gg$ generated by all $e_i$ (respectively $h_i$ and $f_i$). We will assume henceforth that for any $i\in I$, the action of $e_i$ on any $\nn_+$-module is \emph{locally nilpotent}. In other words, if $M$ is an $\nn_+$-module, we will assume that for each $x\in M$ and $i\in I$, there exists some $n\ge 0$ such that $e_i^n(x)=0$. Note that every $\gg$-module is also a $\nn_+$-module, so we are assuming these are ``bounded above" as well. We do not assume the same for the action of $f_i$. For an $\nn_+$-module $M$, we designate $M^+:=\{m\in M\ |\ e_i(m)=0\ \forall i\in I\}$, the set of \emph{highest weight vectors}. If $A$ is an $\nn_+$-module algebra, then $A^+$ is an $\nn_+$-module subalgebra. For $n\in \ZZ_{\ge 0}$ and $i\in I$, we will use the notation $e_i^{(n)}=\frac{1}{n!}e_i^n$. 

Suppose $M$ is an $\mathfrak{n}_+$-module algebra. For each $i\in I$ and $x\in M\setminus\{0\}$, set $\ell_i(x)=\max\{\ell\in\ZZ_{\ge0}\ |\ e_i^\ell(x)\ne 0\}$ and $e_i^{(top)}(x)=e_i^{(\ell_i(x))}(x)$. 
Given ${\bf i}\in I^m$ for some $m\ge 0$, and $x\in M\setminus\{0\}$, we also use the shorthand 
	\[e_{\bf i}^{(top)}(x)=e_{i_m}^{(top)}e_{i_{m-1}}^{(top)}\cdots e_{i_1}^{(top)}(x)\]
and define $\nu_{\bf i}:M\setminus \{0\}\to \ZZ_{\ge 0}^m$, $x\mapsto(a_1,a_2,\ldots,a_m)$ by the following:
	\[a_k=\ell_{i_k}(e_{i_{k-1}}^{(top)}e_{i_{k-2}}^{(top)}\cdots e_{i_1}^{(top)}(x)).\]
Lastly, for ${\bf j}=(j_1,\ldots,j_m)\in \ZZ_{\ge 0}^m$, we set $e_{\bf i}^{({\bf j})}:=e_{i_m}^{(j_m)}e_{i_{m-1}}^{(j_{m-1})}\cdots e_{i_1}^{(j_1)}$.

\begin{definition}
	Let $A$ be an $\nn_+$-module algebra, $w\in W$, and ${\bf i}\in R(w)$. If $e_{\bf i}^{(top)}(x)\in A^+$ for all $x\in A\setminus\{0\}$, then we say $A$ is \emph{{\bf i}-adapted}. We say a basis $\mathcal{B}$ for $A$ is an \emph{{\bf i}-adapted basis} if
	\begin{enumerate}
		\item $e_{\bf i}^{(top)}(b)=1$ for all $b\in \mathcal{B}$.
		\item The restriction of $\nu_{\bf i}$ to $\mathcal{B}$ is an injective map $\mathcal{B}\hookrightarrow \ZZ_{\ge0}^m$, where $m$ is the length of $w$.
	\end{enumerate}
If there exists any $w\in W$ and ${\bf i}\in R(w)$ so that $A$ is {\bf i}-adapted, then we say more generally that $A$ is \emph{adapted}.
\end{definition}

As in the quantum case, if $A_0$ possesses an ${\bf i}$-adapted basis for some ${\bf i}\in R(w)$ and is a ``large enough" $\nn_+$-module subalgebra of $A$, then $A$ is factorizable over $A_0$. The following theorem makes this precise.

\begin{theorem}
\label{thm:factoring}
Let $A$ be an $\mathfrak{n}_+$-module algebra. Suppose $A_0$ be an $\mathfrak{n}_+$-module subalgebra of $A$ possessing an {\bf i}-adapted basis $\mathcal{B}$ for some reduced $\ii$. Then
\begin{enumerate}
	\item The restriction $\mu:A^+\otimes A_0\to A$ of the multiplication in $A$ is an injective homomorphism of $\nn_+$-modules.
	\item The map $\mu$ is an isomorphism if and only if $A$ is {\bf i}-adapted and $\nu_{\bf i}(A\setminus\{0\})=\nu_{\bf i}(A_0\setminus\{0\})$.
\end{enumerate}
\end{theorem}

Theorem \ref{thm:factoring} demonstrates a close relationship between being {\bf i}-adapted and being factorizable over an $\nn_+$-module subalgebra possessing an {\bf i}-adapted basis. The following theorem explores this relationship from a different angle.

\begin{theorem}
\label{thm:classify}
	Let $A$ be an ${\bf i}$-adapted $\nn_+$-module algebra for some reduced $\ii$ and suppose $A_0$ is an $\nn_+$-module subalgebra of $A$. Then $A$ is factorizable over $A_0$ if and only if $A_0$ possesses an ${\bf i}$-adapted basis and $\nu_{\bf i}(A_0\setminus\{0\})=\nu_{\bf i}(A\setminus\{0\})$.
\end{theorem}

The proofs of Theorems \ref{thm:factoring} and \ref{thm:classify} are nearly identical to those of Theorems \ref{thm:qfactoring} and \ref{thm:qclassify}, respectively, so we do not replicate them here. We now restrict our focus to a specific $\gg$-module algebra, $\CC[U]$. Actually, $\CC[U]$ is a specialization of $\CC_q[U]$ to $q=1$. This is accomplished as follows. 

It is well-known (see, e.g. \cite[Section 4]{berensteingreenstein}) that $\CC_q[U]$ admits a form $\underline{\CC_q[U]}$ over $\mathbb{A}=\ZZ[q,q^{-1}]$ which has both a PBW-basis and dual canonical basis. That is, the structure constants of the aforementioned bases belong to $\mathbb{A}$. It is also well-known (see, e.g., \cite[Section 3.3]{berensteingreenstein}) that $E_i^{(n)}(\underline{\CC_q[U]})\subset \underline{\CC_q[U]}$ for all $i\in I$ and $n\in \ZZ_{\ge 0}$. In particular, the quotient of $\underline{\CC_q[U]}$ by the ideal $(q-1)$ generated by $q-1$ is a commutative algebra canonically isomorphic to $\ZZ[U]$. Tensoring by $\CC$, we obtain $\CC[U]$ as the classical limit of $\CC_q[U]$. The action of $E_i$ specializes to the derivations which generate the action of $\nn_+$ on $\CC[U]$.

This in particular implies the well-known fact that $\CC[U]$ is a Poisson algebra with the Poisson bracket given by 
\[\{f,g\}=\dfrac{\tilde{f}\tilde{g}-\tilde{g}\tilde{f}}{q-1} \mod (q-1)\]
for all $f,g\in \CC[U]$, where $\tilde{f}$ and $\tilde{g}$ denote any representatives of $f$ and $g$, respectively, modulo $(q-1)$. Since $\CC_q[U]$ is generated by $\{x_i\ |\ i\in I\}$, $\CC[U]$ has Poisson generators which we denote by slight abuse of notation $\{x_i\ |\ i\in I\}$. The quantum Serre relations for the quantum $x_i$ imply the following relations for the ``classical" versions
	\[\varepsilon(i,j,1-c_{i,j})=0  \text{ if }i\ne j,\]
where $\varepsilon(i,j,n)$ is defined inductively by $\varepsilon(i,j,0)=x_j$, $\varepsilon(i,j,n+1)=\{x_i,\varepsilon(i,j,n)\}-d_i(c_{i,j}+2n)x_i\varepsilon(i,j,n)$.
The $\gg$-module structure on $\CC[U]$ is summarized in the following equations:
	\begin{align*}
		h_i(x_j)&=-c_{i,j}x_j\quad \forall i,j\in I\\
		h_i(\{x,y\})&=\{h_i(x),y\}+\{x,h_i(y)\}\quad \forall i\in I,\ x,y\in \CC[U]\\
		e_i(x_j)&=\delta_{i,j} \quad \forall i,j\in I\\
		e_i(\{x,y\})&=\{e_i(x),y\}+\{x,e_i(y)\}+d_ie_i(x)h_i(y)-d_ih_i(x)e_i(y)\quad \forall i\in I,\ x,y\in \CC[U]\\
		f_i(x)&=\frac{1}{2d_i}\{x_i,x\}+\frac{1}{2}x_ih_i(x)\quad \forall i\in I,\ x\in \CC[U].
	\end{align*}


\begin{remark}
\label{rem:nilpotentf}
	Comparing the defining relations of $\CC[U]$ with the action of $f_i$ thereon, one sees that \linebreak$f_i(x_i)=-x_i^{2}$ and $\varepsilon(i,j,n)=(2d_i)^nf_i^n(x_j)$, so that $f_i^{1-c_{i,j}}(x_j)=0$ if $i\ne j$. Observe also that \linebreak$\{x_i,x\}=2d_if_i(x)-d_ix_ih_i(x)$ for $x\in \CC[U]$ and $i\in I$. It follows that $\CC[U]$ is generated as a $\gg$-module algebra by $\{x_i\ |\ i\in I\}$.
\end{remark}

\begin{example}
\label{ex:mat}
Consider a $3\times 2$ matrix with complex coefficients:
	$A=\begin{pmatrix*}
			a_{1,1} & a_{1,2}\\
			a_{2,1} & a_{2,2}\\
			a_{3,1} & a_{3,2}
		\end{pmatrix*}$.
		
If $a_{1,1}\ne0$ and $a_{1,1}a_{2,2}-a_{1,2}a_{2,1}\ne0$, then $A$ has Gauss factorization

	\[A=\begin{pmatrix*}
			1 & 0 & 0\\
			\frac{a_{2,1}}{a_{1,1}} & 1 & 0\\
			\frac{a_{3,1}}{a_{1,1}} & \frac{a_{1,1}a_{3,2}-a_{1,2}a_{3,1}}{a_{1,1}a_{2,2}-a_{1,2}a_{2,1}} & 1
		\end{pmatrix*}
		\begin{pmatrix*}
			a_{1,1} & a_{1,2} \\
			0 & \frac{a_{1,1}a_{2,2}-a_{1,2}a_{2,1}}{a_{1,1}}\\
			0 & 0
		\end{pmatrix*}.\]
		
Denote by $x_{i,j}$, the $(i,j)$-th coordinate function in $\CC[Mat_{3,2}]$, i.e. $x_{i,j}(A)=a_{i,j}$. The Gauss factorization above implies that upon localization of $\CC[Mat_{3,2}]$ by the principal minors $x_{1,1}$ and $\Delta_2=x_{1,1}x_{2,2}-x_{1,2}x_{2,1}$, we obtain an isomorphism of algebras
	\[\CC[Mat_{3,2}][x_{1,1}^{-1}, \Delta_2^{-1}]\cong \CC\left[\frac{x_{2,1}}{x_{1,1}},\frac{x_{3,1}}{x_{1,1}},\frac{x_{1,1}x_{3,2}-x_{1,2}x_{3,1}}{\Delta_2}\right]\otimes \CC\left[x_{1,1}^{\pm1},x_{1,2},\Delta_2^{\pm1}\right]\]
	
	The natural action of $\sl_3(\CC)$ extends to the localized algebra and a short examination verifies that 
		\[\left(\CC[Mat_{3,2}][x_{1,1}^{-1}, \Delta_2^{-1}]\right)^+=\CC\left[x_{1,1}^{\pm1},x_{1,2},\Delta_2^{\pm1}\right]~\text{and}~\CC\left[\frac{x_{2,1}}{x_{1,1}},\frac{x_{3,1}}{x_{1,1}},\frac{x_{1,1}x_{3,2}-x_{1,2}x_{3,1}}{\Delta_2}\right]\cong \CC[U],\]
	where the isomorphism is an isomorphism of $\sl_3(\CC)$-module algebras and the generators $x_1$ and $x_2$ of $\CC[U]$ are mapped to by $\dfrac{x_{2,1}}{x_{1,1}}$ and $\dfrac{x_{1,1}x_{3,2}-x_{1,2}x_{3,1}}{\Delta_2}$, respectively.
\end{example}

\begin{example}
\label{ex:qmat}
	Recall that $\CC_q[Mat_{3,2}]$ is generated by $\{x_{i,j}\ |\ 1\le i\le 3,\ 1\le j\le 2\}$, subject to relations
		\begin{align*}
			x_{k,j}x_{i,j}&=qx_{i,j}x_{k,j} \text{ if }i<k,\\
			x_{i,k}x_{i,j}&=qx_{i,j}x_{i,k} \text{ if }j<k,\\
			x_{k,\ell}x_{i,j}&=x_{i,j}x_{k,\ell}\text{ if }i<k\text{ and }j>\ell,\\
			x_{k,\ell}x_{i,j}&=x_{i,j}x_{k,\ell}+(q-q^{-1})x_{i,\ell}x_{k,j}\text{ if }i<k\text{ and }j<\ell.
		\end{align*}		
Then
$\CC_q[Mat_{3,2}][x_{1,1}^{-1}, \Delta_2^{-1}]\cong  \CC_q\left[x_{1,1}^{\pm1},x_{1,2},\Delta_2^{\pm1}\right]
			\otimes \CC_q\left[x_{1,1}^{-1}x_{2,1},x_{1,1}^{-1}x_{3,1},\Delta_2^{-1}(x_{1,1}x_{3,2}-q^{-1}x_{1,2}x_{3,1})\right]$, 
where $\Delta_2=x_{1,1}x_{2,2}-q^{-1}x_{1,2}x_{2,1}$ and $\CC_q[-]$ denotes the subalgebra of $\CC_q[Mat_{3,2}][x_{1,1}^{-1}, \Delta_2^{-1}]$ generated by those elements appearing inside the brackets.
	
	The natural action of $U_q(\sl_3(\CC))$ extends to the localized algebra and a short examination verifies that 
		\[\left(\CC_q[Mat_{3,2}][x_{1,1}^{-1}, \Delta_2^{-1}]\right)^+=\CC_q\left[x_{1,1}^{\pm1},x_{1,2},\Delta_2^{\pm1}\right]\quad \text{and}\]
		\[\CC_q\left[x_{1,1}^{-1}x_{2,1},x_{1,1}^{-1}x_{3,1},\Delta_2^{-1}(x_{1,1}x_{3,2}-q^{-1}x_{1,2}x_{3,1})\right]\cong \CC[U],\]
	where the isomorphism is an isomorphism of $U_q(\sl_3(\CC))$-module algebras and the generators $x_1$ and $x_2$ of $\CC_q[U]$ are mapped to by $x_{1,1}^{-1}x_{2,1}$ and $\Delta_2^{-1}(x_{1,1}x_{3,2}-q^{-1}x_{1,2}x_{3,1})$, respectively.
\end{example}

Now, since $\CC_q[U]$ possessed an $\ii$-adapted basis for every $\ii\in R(\wnot)$ and the action of $e_i$ on $\CC[U]$ is induced by that of $E_i$ on $\CC_q[U]$, it follows that $\CC[U]$ possesses an $\ii$-adapted basis for each $\ii\in R(\wnot)$. Combining this fact with Theorem \ref{thm:factoring} leads us to the following corollary.

\begin{corollary}
\label{cor:highest}
	Let $A$ be a $\gg$-module algebra containing $\CC[U]$ as a $\gg$-module subalgebra. The map \linebreak$\mu:A^+\otimes \CC[U]\to A$ as in Theorem \ref{thm:factoring} is an isomorphism of $\nn_+$-modules if and only if there exists a $\gg$-module algebra $A'$ which contains $A$ as a $\gg$-module subalgebra and is generated by $(A')^+$ as a $\gg$-module algebra.
\end{corollary}

The proof of the ``if" part of Corollary \ref{cor:highest} is very similar to the proof of Corollary \ref{cor:qhighest} so we do not reproduce it here. The ``only if" part will be a very easy consequence of the discussion at the end of Section \ref{sect:classical}, so we will address it there. Also, just as in the quantum setting, Corollary \ref{cor:highest} leads to a proof of Theorem \ref{thm:G/U}, which is nearly identical to that of Theorem \ref{thm:qG/U}, so we will not include it here. We do, however, note that instead of two families of quantities as arose in the proof of Corollary \ref{cor:qhighest}, only one arises in the proof of the ``if" part of Corollary \ref{cor:highest}:
\[f_i(a)-h_i(a)x_i\]
where $i\in I$ and $a\in A'$. As in the quantum case, if $a\in A^+$, then $f_i(a)-h_i(a)x_i\in A^+$. Once again it is natural to ask what relations the family of operators $f_i-m_ih_i$ satisfies, where $m_i$ denotes multiplication by $x_i$. Or to put it another way, do these operators indicate the action of a known algebra or Lie algebra which is somehow related to $\gg$? Again, we can answer in the affirmative, which the following theorem summarizes.

\begin{theorem}
\label{thm:b-}
	Let $A$ be a $\gg$-module algebra containing $\CC[U]$ as a $\gg$-module subalgebra. Then $A$ has another structure of a $\bb_-$-module algebra with action given by the formulas
		\[h_i\rhd a=h_i(a),\quad f_i\rhd a=f_i(a)-h_i(a)x_i.\]
	In particular, the subalgebra $A^+$ is invariant under this $\bb_-$ action and is therefore a $\bb_-$-module subalgebra.
\end{theorem}

Theorem \ref{thm:b-} is proved in Section \ref{pf:b-,g}. As in the quantum case, Theorem \ref{thm:b-} is in some sense a statement about the existence of a functor. To make this precise, we introduce a category whose objects bear properties similar to those found in Theorem \ref{thm:factoring}.

\begin{definition}
	Let $\mathcal{C}_\gg$ be the category whose objects consist of pairs $(A,\varphi_A)$, where 
	\begin{itemize}
		\item $A$ is an adapted $\gg$-module algebra such that $\nu_{\bf i}(A\setminus \{0\})=\nu_{\bf i}(\CC[U]\setminus \{0\})$ for all ${\bf i}\in R(\wnot)$;
		\item $\varphi_A:\CC[U]\hookrightarrow A$ is an embedding of $\gg$-module algebras.
	\end{itemize}
	A morphism $(A,\varphi_A)\to (B,\varphi_B)$ in $\mathcal{C}_\gg$ is a homomorphism of $\gg$-module algebras $\psi:A\to B$ such that $\psi\circ \varphi_A=\varphi_B$.
\end{definition}

Given a homomorphism of $\gg$-module algebras $\psi:A\to B$, it follows that $\psi(A^+)\subseteq B^+$, so $\psi|_{A^+}$ may be thought of as a map of $\CC$-algebras $A^+\to B^+$. If $\psi$ is a morphism in $\mathcal{C}_\gg$, $(A,\varphi_A)\to (B,\varphi_B)$, then actually $\psi|_{A^+}$ is a homomorphism of $\bb_-$-module algebras, where the $\bb_-$-module structure is the one given in Theorem \ref{thm:b-}. As a consequence, we have the following corollary.

\begin{corollary}
	There is a functor $(-)^+:\mathcal{C}_\gg\to \bb_--{\bf ModAlg}$ (denoted $P\circ F$ in Section \ref{sect:Introduction}) which assigns to an object $(A,\varphi_A)$ of $\mathcal{C}_\gg$ its subalgebra of highest weight vectors $A^+$, equipped with the $\bb_-$-module algebra structure of Theorem \ref{thm:b-}. The functor $(-)^+$ is given on morphisms by restriction.
\end{corollary}

Again, we hope for a quasi-inverse functor for $(-)^+$, but we must first state the following theorem.

\begin{theorem}
\label{thm:g}
	If $A$ is a $\bb_-$-module algebra, then the usual tensor product of commutative algebras $A\otimes \CC[U]$ has the structure of a $\gg$-module algebra determined by:
	\begin{align*}
		h_i\rhd(a\otimes x)&=h_i(a)\otimes x+a\otimes h_i(x)\\
		e_i\rhd(a\otimes x)&=a\otimes e_i(x)\\
		f_i\rhd(a\otimes x)&=f_i(a)\otimes x+h_i(a)\otimes x_ix+a\otimes f_i(x).
	\end{align*}
\end{theorem}

Theorem \ref{thm:g} will be proved in Section \ref{pf:b-,g}. Since the action of each $e_i$ is completely described on the $\CC[U]$ factor and $\CC[U]$ is adapted, the following corollary is almost immediate.

\begin{corollary}
	There is a functor $(-)\otimes \CC[U]:\bb_--{\bf ModAlg}\to \mathcal{C}_\gg$ which assigns to a $\bb_-$-module algebra $A$, the pair $(A\otimes \CC[U],1\otimes \text{id})$, where $A\otimes \CC[U]$ is given the $\gg$-module algebra structure of Theorem \ref{thm:g}. The functor $(-)\otimes \CC[U]$ is given on morphisms by $\psi\mapsto \psi\otimes \text{id}$.
\end{corollary}

As promised, $(-)\otimes \CC[U]$ is the desired quasi-inverse for $(-)^+$.

\begin{theorem}
\label{thm:equiv}
	The functors $(-)^+:\mathcal{C}_\gg\to \bb_--{\bf ModAlg}$ and $(-)\otimes \CC[U]:\bb_--{\bf ModAlg}\to \mathcal{C}_\gg$ are quasi-inverses of each other and thus provide equivalences of categories.
\end{theorem}

The proof of Theorem \ref{thm:equiv} is very similar to that of \ref{thm:qequiv}, so we do not include it here. Now, it is well-known that if $A$ and $B$ are $\gg$-module algebras, then so is $A\otimes B$. Here $A\otimes B$ has the na\"ive multiplication $(a\otimes b)(a'\otimes b')=aa'\otimes bb'$. So if $(A,\varphi_A)$ and $(B,\varphi_B)$ are objects of $\mathcal{C}_\gg$, then $A\otimes B$ is a $\gg$-module algebra. Furthermore, it is obvious that $1\otimes \varphi_B$ and $\varphi_A\otimes 1$ are injections $\CC[U]\hookrightarrow A{\otimes} B$. However, it is not immediately obvious that $A{\otimes} B$ is adapted with $\nu_{\bf i}(A{\otimes} B\setminus\{0\})=\nu_{\bf i}(\CC[U]\setminus\{0\})\ \forall {\bf i}\in R(\wnot)$. Nevertheless, this is the case and the following proposition asserts as much and is proved in Section \ref{pf:tensor}.

\begin{proposition}
\label{prop:tensor}
	If $(A,\varphi_A)$ and $(B,\varphi_B)$ are objects of $\mathcal{C}_\gg$, then $(A\otimes B, 1\otimes \varphi_B)$ and $(A\otimes B,\varphi_A\otimes 1)$ are objects of $\mathcal{C}_\gg$ as well. 
\end{proposition}



Theorem \ref{thm:equiv} and Proposition \ref{prop:tensor} allow us to define two ``fusion" products on the category of $\bb_-$-module algebras, namely the following:
	\begin{align*}
		A*B&:=((A\otimes \CC[U]){\otimes} (B\otimes \CC[U]),1\otimes 1\otimes 1\otimes \text{id})^+\\
		A\star B&:=((A\otimes \CC[U]){\otimes} (B\otimes \CC[U]),1\otimes \text{id}\otimes 1\otimes 1)^+.
	\end{align*}

Unfortunately, as in the quantum case, these fusion products are not monoidal as there is no unit object. Furthermore, the objects $(A\otimes B,1\otimes \varphi_B)$ and $(A\otimes B,\varphi_A\otimes 1)$ are not necessarily isomorphic, despite having equal underlying $\gg$-module algebras. However, given $\bb_-$-module algebras $A$ and $B$, we of course have the natural $\bb_-$-module algebra structure on $A\otimes B$ satisfying
$$		(a\otimes b)(a'\otimes b')=aa'\otimes bb',~h_i(a\otimes b)=h_i(a)\otimes b+a\otimes h_i(b),~
		f_i(a\otimes b)=f_i(a)\otimes b+a\otimes f_i(b)$$
for $a,a'\in A$, $b,b'\in B$, and $i\in I$. This induces a more symmetric fusion product on $\mathcal{C}_\gg$:
	\begin{equation}
	\label{eq:fusion}
		(A,\varphi_A)\diamond (B,\varphi_B):=(((A,\varphi_A)^+\otimes (B,\varphi_B)^+)\otimes \CC[U],1\otimes 1\otimes \text{id}).
	\end{equation}

\begin{remark}
\label{rmk:unitissue}
	Since $\bb_-$-{\bf ModAlg} is a monoidal category with unit object $\CC$, the product $\diamond$ makes $\mathcal{C}_\gg$ into a monoidal category with unit object $\CC[U]=\CC\otimes\CC[U]$. Given objects $(A,\varphi_A)$ and $(B,\varphi_B)$ of $\mathcal{C}_\gg$, $(A,\varphi_A)\diamond(B,\varphi_B)$ is easily observed to be isomorphic to the quotient object
	\[(A\!\!\!\underset{\CC[U]}{\otimes}\! \!\! B,\pi\circ(1\otimes \varphi_B))\cong(A\!\!\!\underset{\CC[U]}{\otimes}\! \!\! B,\pi\circ(\varphi_A\otimes 1))\]
	where $A\!\!\!\underset{\CC[U]}{\otimes}\! \!\! B$ is the quotient of the $\gg$-module algebra $A\otimes B$ by the ideal generated by all elements of the form $\varphi_A(x)\otimes 1-1\otimes \varphi_B(x)$ for $x\in \CC[U]$ and $\pi:A\otimes B\to A\!\!\!\underset{\CC[U]}{\otimes}\! \!\! B$ is the quotient map.
\end{remark}
	
This natural structure also results in a very easy proof of the ``only if" part of Corollary \ref{cor:highest}. 
Denote by $\CC[T]$ the algebra with basis $\{v_\lambda\ |\ \lambda\in \Lambda\}$ (where $v_0=1$) and multiplication $v_\lambda v_\mu=v_{\lambda+\mu}$. We make it into a $\bb_-$-module algebra with $\bb_-$-module structure given by $h_i(v_\lambda)=(\alpha_i^\vee,\lambda)v_\lambda$ and $f_i(v_\lambda)=0$. Suppose $\mu:A^+\otimes \CC[U]\to A$ as in Theorem \ref{thm:factoring} is an isomorphism. We give $A^+$ the structure of a $\bb_-$-module algebra as in Theorem \ref{thm:b-}. Then Theorem \ref{thm:g} allows us to make $(A^+\otimes \CC[T])\otimes \CC[U]$ into a $\gg$-module algebra. Now $A$ is clearly a $\gg$-module subalgebra and 
\begin{align*}
	((1\otimes v_{-\omega_i})\otimes 1)\left[f_i\rhd ((1\otimes v_{\omega_i})\otimes 1)\right]&=((1\otimes v_{-\omega_i})\otimes 1)\left((1\otimes h_i(v_{\omega_i}))\otimes x_i\right)\\
		&=(\alpha_i^\vee, \omega_i)(1\otimes v_{-\omega_i}v_{\omega_i}\otimes x_i)\\
		&=1\otimes 1\otimes x_i
\end{align*}
showing that $(A^+\otimes \CC[T])\otimes \CC[U]$ is generated by $(A^+\otimes \CC[T])\otimes \CC=((A^+\otimes \CC[T])\otimes \CC[U])^+$ as a $\gg$-module algebra and proving the ``only if" part of Corollary \ref{cor:highest}.


\section{Proofs}
\label{sect:Proofs}

In many proofs, we will use the fact that $\ZZ_{\ge 0}^m$ is well-ordered by the lexicographic order. For given $w\in W$, ${\bf i}\in R(w)$, and $U_q(\bb_+)$-module $M$, we have that $\nu_{\bf i}(M)$ is well-ordered, allowing us to induct on $\nu_{\bf i}(x)$ for $x\in M$.

\subsection{Proof of Theorem \ref{thm:qfactoring}}
\label{pf:qfactoring}

For ${\bf j}\in \nu_{\bf i}(\mathcal{B})$, let $b_{\bf j}$ be the unique element of $\mathcal{B}$ such that $\nu_{\bf i}(b_{\bf j})={\bf j}$.

\begin{enumerate}
\item We first observe that since $A$ is a $U_q(\gg)$-module algebra and $A^+$ and $A_0$ are $U_q(\bb_+)$-submodules, $\mu$ is a homomorphism of $U_q(\bb_+)$-modules. Hence we simply show that $\mu$ is injective.
	Now each nonzero element $a\in A^+\otimes A_0$ can be written \[a=\sum_{k=1}^n a_k\otimes b_{{\bf j}_k}\] for some $n>0$, $a_k\in A^+\setminus \{0\}$, and ${\bf j}_k\in \nu_{\bf i}(A_0\setminus\{0\})$. We may assume ${\bf j}_k<{\bf j}_{l}$ if $1\le k<l\le n$ so that 
	\[E_{\bf i}^{{\bf j}_l}(b_{{\bf j}_k})=\begin{cases} 0 & \text{if } k<l\\ 1 & \text{if } k=\ell\end{cases}.\] Suppose for the sake of contradiction that $\mu(a)=0$. Then
$$		0=E_{\bf i}^{({\bf j}_n)}(\mu(a))=\mu\left(E_{\bf i}^{({\bf j}_n)}(a)\right)=\mu(a_n\otimes 1)=a_n$$
	which is a contradiction. Hence $\mu(a)=0$ if and only if $a=0$, showing that $\mu$ is injective.

\item	($\Rightarrow$) Suppose $\mu$ is an isomorphism. Given nonzero $a\in A$, write
	\[a=\mu\left(\sum_{k=1}^n a_k\otimes b_{{\bf j}_k}\right)\]
	as in (1). Then, since $\mu$ is injective, it is clear that $\nu_{\bf i}(a)={\bf j}_n=\nu_{\bf i}(b_{{\bf j}_n})$, showing that \linebreak{$\nu_{\bf i}(A\setminus\{0\})\subseteq\nu_{\bf i}(A_0\setminus\{0\})$.} But since $A_0\subseteq A$, it follows that $\nu_{\bf i}(A\setminus\{0\})=\nu_{\bf i}(A_0\setminus\{0\})$. Also, as seen above 
	\[E_{\bf i}^{(top)}(a)=E_{\bf i}^{({\bf j}_n)}(a)=\mu(a_n\otimes 1)=a_n\in A^+,\] 
	so we see that $A$ is ${\bf i}$-adapted.\\
($\Leftarrow$) Suppose that $A$ is {\bf i}-adapted and $\nu_{\bf i}(A\setminus\{0\})=\nu_{\bf i}(A_0\setminus\{0\})$. By (1), we already know that $\mu$ is an injective $U_q(\bb_+)$-module homomorphism. Hence we simply use induction to show that $\mu$ is surjective. We first note that since $A$ is {\bf i}-adapted, if $\nu_{\bf i}(a)=(0,0,\ldots,0)$, then $a=E_{\bf i}^{(top)}(a)\in A^+$. In other words,
		\[\{a\in A\setminus\{0\}\ |\ \nu_{\bf i}(a)=(0,0,\ldots,0)\}=A^+\setminus \{0\}\subset \mu(A^+\otimes A_0).\] 
	Let $a\in A\setminus\{0\}$ and suppose $a'\in \mu(A^+\otimes A_0)$ for all $a'\in A\setminus\{0\}$ such that $\nu_{\bf i}(a')<\nu_{\bf i}(a)$. We have 
	\[E_{\bf i}^{(\nu_{\bf i}(a))}(a-\mu(E_{\bf i}^{(top)}(a)\otimes b_{\nu_{\bf i}(a)}))=0.\]
	Hence either $a-\mu(E_{\bf i}^{(top)}(a)\otimes b_{\nu_{\bf i}(a)})=0$ or $\nu_{\bf i}(a-\mu(E_{\bf i}^{(top)}(a)\otimes b_{\nu_{\bf i}(a)}))<\nu_{\bf i}(a)$. In the former case, $a\in \mu(A^+\otimes A_0)$. In the latter case, $a-\mu(E_{\bf i}^{(top)}(a)\otimes b_{\nu_{\bf i}(a)})\in \mu(A^+\otimes A_0)$ and so
		\[a=(a-\mu(E_{\bf i}^{(top)}(a)\otimes b_{\nu_{\bf i}(a)}))+\mu(E_{\bf i}^{(top)}(a)\otimes b_{\nu_{\bf i}(a)})\in \mu(A^+\otimes A_0).\]
	So we have shown that $a\in \mu(A^+\otimes A_0)$. By induction, $\mu$ is surjective. Hence $\mu$ is an isomorphism.
\end{enumerate}\qed

\subsection{Proof of Theorem \ref{thm:qclassify}}
\label{pf:qclassify}

$(\Leftarrow)$ Suppose $A_0$ possesses an {\bf i}-adapted basis and $\nu_{\bf i}(A_0\setminus\{0\})=\nu_{\bf i}(A\setminus\{0\})$. Then by Theorem \ref{thm:qfactoring}, $\mu:A^+\otimes A_0\to A$ is an isomorphism of $U_q(\bb_+)$-modules.

$(\Rightarrow)$ Suppose $\mu:A^+\otimes A_0\to A$ is an isomorphism of $U_q(\bb_+)$-modules. Hence we must have $(A_0)^+=\CC(q)$ or else $\mu$ would fail to be injective. Also, since $A$ is {\bf i}-adapted, $A_0$ is as well. Now for each ${\bf j}\in \nu_{\bf i}(A_0\setminus\{0\})$ choose $b_{\bf j}\in A_0\setminus\{0\}$ such that $E_{\bf i}^{(top)}(b_{\bf j})=1$ and $\nu_{\bf i}(b_{\bf j})={\bf j}$ (note that $b_{(0,\cdots,0)}=1$). We claim that $\mathcal{B}=\{b_{\bf j}\ |\ {\bf j}\in \nu_{\bf i}(A_0\setminus \{0\})\}$ is an ${\bf i}$-adapted basis for $A_0$. To prove that $\mathcal{B}$ is linearly independent and spans $A_0$, we mimic the proofs that $\mu$ is injective and surjective (respectively) in Theorem \ref{thm:qfactoring}.
	Suppose \[\sum_{k=1}^n r_kb_{{\bf j}_k}=0\] for some $r_k\in \CC(q)$ and ${\bf j}_k\in\nu_{\bf i}(A_0\setminus\{0\})$ such that ${\bf j}_k<{\bf j}_l$ if $k<l$. Then
		\[0=E_{\bf i}^{({\bf j}_n)}\left(\sum_{k=1}^n r_kb_{{\bf j}_k}\right)=r_n.\]
	By induction, each $r_k=0$. It follows that $\mathcal{B}$ is linearly independent. 
	Note that \[\{a\in A_0\setminus \{0\}\ |\ \nu_{\bf i}(a)=(0,0,\ldots,0)\}=(A_0)^+\setminus\{0\}=\CC(q)^{\times}\subseteq \text{span}_{\CC(q)}(\mathcal{B}).\]
	Let $a\in A_0\setminus\{0\}$ and suppose $a'\in \text{span}_{\CC(q)}(\mathcal{B})$ for all $a'\in A_0\setminus\{0\}$ such that $\nu_{\bf i}(a')<\nu_{\bf i}(a)$. We have
	\[E_{\bf i}^{(\nu_{\bf i}(a))}(a-E_{\bf i}^{(top)}(a)b_{\nu_{\bf i}(a)})=0.\]
	Hence either $a-E_{\bf i}^{(top)}(a)b_{\nu_{\bf i}(a)}=0$ or $\nu_{\bf i}(a-E_{\bf i}^{(top)}(a)b_{\nu_{\bf i}(a)})<\nu_{\bf i}(a)$. In the former case $a\in \text{span}_{\CC(q)}(\mathcal{B})$. In the latter case $a-E_{\bf i}^{(top)}(a)b_{\nu_{\bf i}(a)}\in \text{span}_{\CC(q)}(\mathcal{B})$ and so 
	\[a=(a-E_{\bf i}^{(top)}(a)b_{\nu_{\bf i}(a)})+E_{\bf i}^{(top)}(a)b_{\nu_{\bf i}(a)}\in \text{span}_{\CC(q)}(\mathcal{B}).\]
	So we have shown that $a\in \text{span}_{\CC(q)}(\mathcal{B})$. By induction, $\mathcal{B}$ spans $A_0$. Hence we have shown that $\mathcal{B}$ is a basis for $A_0$. By construction, it is in fact an {\bf i}-adapted basis for $A_0$.
	
	In light of $\mathcal{B}$'s existence, a typical element of $A$ is of the form $\mu\left(\sum\limits_{k=1}^n a_k\otimes b_{{\bf j}_k}\right)$
	for some $a_k\in A^+$ and ${\bf j}_k\in \nu_{\bf i}(A_0\setminus\{0\})$. It is now clear that $\nu_{\bf i}\left(\mu\left(\sum\limits_{k=1}^n a_k\otimes b_{{\bf j}_k}\right)\right)=\max\{{\bf j}_k\ |\ k=1,\ldots,n\}$
so that \linebreak$\nu_{\bf i}(A_0\setminus\{0\})\supseteq \nu_{\bf i}(A\setminus\{0\})$. Since $A_0\subseteq A$, we have $\nu_{\bf i}(A_0\setminus\{0\})\subseteq\nu_{\bf i}(A\setminus\{0\})$ and so $\nu_{\bf i}(A_0\setminus\{0\})=\nu_{\bf i}(A\setminus\{0\})$. 
		\qed

\subsection{Proof of Proposition \ref{prop:adapted}}
\label{pf:adapted}
We have already observed that for any ${\bf i}\in R(\wnot)$, the restriction of $\nu_{\bf i}$ to $\mathcal{B}^{dual}$ is an injective map $\mathcal{B}^{dual}\hookrightarrow \ZZ_{\ge 0}^m$, where $m$ is the length of $\wnot$. Hence it suffices to show that $E_{\bf i}^{(top)}(b)=1$ for all ${\bf i}\in R(\wnot)$ and $b\in \mathcal{B}^{dual}$. To do this we need the following lemma.

\begin{lemma}
\label{lem:ewtop}
Given $w\in W$ and ${\bf i, i'}\in R(w)$, $E_{\bf i}^{(top)}(b)=E_{\bf i'}^{(top)}(b)$ for all $b\in \mathcal{B}^{dual}$.
\end{lemma}

\begin{proof}
Now $\CC_q[U]$ factors as the product of two subalgebras: $\CC_q[U]=(\CC_q[U]_{>w})(\CC_q[U]_{\le w})$ (see \cite{kimura}, for example, where they are respectively denoted ${\bf U}_q^-(>w,-1)$ and ${\bf U}_q^-(\le w,-1)$). In fact, these subalgebras can be described explicitly as follows. For any reduced word $\ii\in R(\wnot)$ such that $s_{i_1}\cdots s_{i_k}=w$, consider elements $X_1,\cdots,X_m$ as in \cite[Section 4]{berensteingreenstein}, where $m$ is the length of $\wnot$. This choice guarantees that monomials $X^{\bf a}=X_1^{a_1}\cdots X_m^{a_m}$ for ${\bf a}\in \ZZ_{\ge 0}^m$ form a basis for $\CC_q[U]$. It follows that those $X^{\bf a}$ with $a_\ell=0$ for $\ell>k$ form a basis for $\CC_q[U]_{\le w}$ and those $X^{\bf a}$ with $a_\ell=0$ for $\ell\le k$ form a basis for $\CC_q[U]_{>w}$. Since $X_1=x_{i_1}$ and these two subalgebras are orthogonal with respect to Lusztig's pairing (under which multiplication by $x_i$ and action by $E_i$ are adjoint), we obtain the following well-known fact:
	\[E_{i}(\CC_q[U]_{>w})=0\]
for any $i\in I$ such that $\ell(s_iw)<\ell(w)$. In particular, this implies that 
	\[E_{i}(\CC_q[U]_{>w_{i,j}})=E_{j}(\CC_q[U]_{>w_{i,j}})=0,\]
where $w_{i,j}$ is the longest element in the subgroup generated by $s_i$ and $s_j$.
	
It is well-known that any two reduced words for a fixed $w\in W$ are related by a series of rank two relations. Hence it suffices to show the lemma when ${\bf i}$ and ${\bf i'}$ differ by a single rank two relation. But it is also well-known that $E_j^{(top)}(b)\in \mathcal{B}^{dual}$ for all $j\in I$ and $b\in \mathcal{B}^{dual}$. For any ${\bf j}\in R(w)$, the operator $E_{\bf j}^{(top)}$ is by definition just the composition of operators $E_{j_\ell}^{(top)}\cdots E_{j_2}^{(top)}E_{j_1}^{(top)}$, where $\ell$ is the length of $w$. This reduces the problem to the case when $w$ is the longest element of a rank two parabolic subgroup of $W$. We will therefore assume for the rest of the proof that ${\bf i}=(i,j,\ldots)$ and ${\bf i'}=(j,i,\ldots)$ the only two distinct reduced words for $w_{i,j}$. An explicit (and apparently well-known) computation verifies that	
	\[E_{\bf i}^{(top)}(\CC_q[U]_{\le w_{i,j}}\cap \mathcal{B}^{dual})=E_{\bf i'}^{(top)}(\CC_q[U]_{\le w_{i,j}}\cap \mathcal{B}^{dual})=1.\] 
According to \cite[Theorem 3.14]{kimura}, for each $b\in \mathcal{B}^{dual}$, there exist $b'\in \CC_q[U]_{>w_{i,j}}\cap \mathcal{B}^{dual}$, $b''\in \CC_q[U]_{\le w_{i,j}}\cap \nolinebreak\mathcal{B}^{dual}$, and $\xi\in \CC_q[U]$ such that $\nu_{\bf i}(\xi)<\nu_{\bf i}(b)$, $\nu_{\bf i'}(\xi)<\nu_{\bf i'}(b)$, and 
		\[b'b''=b+\xi.\]	
	Hence
$$			E_{\bf i}^{(top)}(b)=E_{\bf i}^{(top)}(b+\xi)=E_{\bf i}^{(top)}(b'b'')=b'E_{\bf i}^{(top)}(b'')=b'.$$
	Likewise, $E_{\bf i'}^{(top)}(b)=b'$, so the lemma is proved.
\end{proof}

In light of Lemma \ref{lem:ewtop}, given $w\in W$ and $b\in \mathcal{B}^{dual}$, we may unambiguously define $E_w^{(top)}(b):=E_{\bf i}^{(top)}(b)$ for any ${\bf i}\in R(w)$. Now given $j\in I$, there exists some ${\bf i}\in R(\wnot)$ such that if ${\bf i}=(i_1,\ldots,i_m)$, then $i_m=j$. It follows that $E_j(E_{\wnot}^{(top)}(b))=0$ for all $j\in I$, i.e. $E_{\wnot}^{(top)}(b)\in (\CC_q[U])^+=\CC(q)$. Since $E_j^{(top)}(b)\in \mathcal{B}^{dual}$ for all $j\in I$ and $b\in \mathcal{B}^{dual}$, it follows that $E_{\wnot}^{(top)}(b)=1$ for $b\in \mathcal{B}^{dual}$. \qed

\subsection{Proof of Corollary \ref{cor:qhighest}}
\label{pf:qhighest}

Let ${\bf i}\in R(\wnot)$. As previously remarked, $\CC_q[U]$ possesses an {\bf i}-adapted basis. Then Theorem \ref{thm:qfactoring} (1) says that $\mu':(A')^+\otimes \CC_q[U]\to A'$ is an injective homomorphism of $U_q(\bb_+)$-modules.
	
	We now show that $\mu'((A')^+\otimes \CC_q[U])$ is a $U_q(\gg)$-module subalgebra of $A'$ and hence is equal to $A'$ by the assumption that $(A')^+$ generates $A'$. To see that $\mu'((A')^+\otimes \CC_q[U])$ is a subalgebra of $A'$, it suffices to show that $x_ia\in \mu'((A')^+\otimes \CC_q[U])$ for $i\in I$ and $a\in (A')^+$. For this, we observe that 
	\[E_i(x_ja-K_j(a)x_j)=\delta_{ij}K_i(a)-\delta_{ij}K_j(a)=0\]
	and hence $x_ja-K_j(a)x_j\in (A')^+$. Then
$$		x_ia=K_i(a)x_i+(x_ia-K_i(a)x_i)\\
			=\mu'(K_i(a)\otimes x_i+(x_ia-K_i(a)x_i)\otimes 1)\\
			\in \mu'((A')^+\otimes \CC_q[U]).$$
So $\mu'((A')^+\otimes \CC_q[U])$ is a subalgebra of $A'$. 

Now we need to show that $\mu'((A')^+\otimes \CC_q[U])$ is closed under the action of $U_q(\gg)$. By Theorem \ref{thm:qfactoring} (1), $\mu'$ is $U_q(\bb_+)$-invariant and hence it suffices to show that $\mu'((A')^+\otimes\CC_q[U])$ is closed under the action of $F_i$ for $i\in I$. Observe that for $a\in (A')^+$ and $x\in \CC_q[U]$, $\mu'(a\otimes x)=ax$, so we will simply compute the action of $F_i$ on such an element. However, before doing so, we note that for $a\in (A')^+$ and $i,j\in I$, we have 
	\[E_i\left(F_j(a)+x_j\frac{K_j^{-2}(a)-a}{q_j-q_j^{-1}}\right)=\delta_{ij}\frac{K_i(a)-K_i^{-1}(a)}{q_i-q_i^{-1}}+\delta_{ij}\frac{K_i^{-1}(a)-K_i(a)}{q_i-q_i^{-1}}=0,\]
showing that $F_j(a)+x_j\dfrac{K_j^{-2}(a)-a}{q_i-q_i^{-1}}\in (A')^+$. Now we compute:

		\begin{align*}
			F_i(ax)&=F_i(a)x+K_i^{-1}(a)F_i(x)\\
			&=\left(F_i(a)+x_i\frac{K_i^{-2}(a)-a}{q_i-q_i^{-1}}\right)x-x_i\frac{K_i^{-2}(a)-a}{q_i-q_i^{-1}}x+K_i^{-1}(a)F_i(x)\\
			&=\mu'\left(\left(F_i(a)+x_i\frac{K_i^{-2}(a)-a}{q_i-q_i^{-1}}\right)\otimes x+K_i^{-1}(a)\otimes F_i(x)\right)-\mu'(1\otimes x_i)\mu'\left(\frac{K_i^{-2}(a)-a}{q_i-q_i^{-1}}\otimes x\right)\\
								&\in \mu'((A')^+\otimes \CC_q[U]).
		\end{align*}
	So $\mu'((A')^+\otimes \CC_q[U])$ is closed under the action of $U_q(\gg)$ and we may conclude that $\mu'((A')^+\otimes \CC_q[U])=A'$. Hence $\mu'$ is an isomorphism. By Theorem \ref{thm:qfactoring} (2), this implies that $A'$ is {\bf i}-adapted and \linebreak$\nu_{\bf i}(A'\setminus\{0\})=\nu_{\bf i}(\CC_q[U]\setminus\{0\})$. We deduce that $A$ is {\bf i}-adapted and \[\nu_{\bf i}(\CC_q[U]\setminus\{0\})\subseteq\nu_{\bf i}(A\setminus\{0\})\subseteq\nu_{\bf i}(A'\setminus\{0\})=\nu_{\bf i}(\CC_q[U]\setminus\{0\}),\] i.e. $\nu_{\bf i}(\CC_q[U]\setminus\{0\})=\nu_{\bf i}(A\setminus\{0\})$. Hence applying Theorem \ref{thm:qfactoring} (2) again, $\mu$ is an isomorphism.
\qed

\subsection{Proof of Theorem \ref{thm:qG/U}}
\label{pf:qG/U}

It is well-known (see, e.g., \cite[Section 6.1]{berensteinrupel}) that $\CC_q[B]:=U_q(\bb_+)^*$ is naturally a $U_q(\gg)$-module algebra and its $U_q(\gg)$-module subalgebra generated by all highest weight vectors $v_\lambda$ for $\lambda\in \Lambda^+$ is isomorphic as a $U_q(\gg)$-module to the direct sum of all simple $V_\lambda$ for $\lambda\in \Lambda^+$ and thus identifies with the $q$-deformation $\CC_q[G/U]$ of $\CC[G/U]$. 
By the construction, the highest weight vectors satisfy $v_\lambda v_\mu=v_{\lambda+\mu}$ in $\CC_q[B]$. Before showing the factorizability of $\CC_q[G/U]$ after localization, we recall the definition of right Ore sets (which allow for Ore localizations and are sometimes also called right denominator sets) for the reader's convenience.

\begin{definition}
	Let $R$ be any unital ring. A submonoid $\mathcal{S}\subset R\setminus \{0\}$ is called a \emph{right Ore set} if the following conditions are satisfied for $r\in R$ and $s\in \mathcal{S}$:
		\begin{enumerate}
			\item $r\mathcal{S}\cap sR\ne \varnothing$.
			\item If $sr=0$, then $\exists s'\in \mathcal{S}$ such that $rs'=0$.
		\end{enumerate}
\end{definition}

Recall (see, e.g., \cite{jordan}) that an element $p$ of a ring $R$ is \emph{normal} if $pR=Rp$. It is immediate (and well-known) that for any ring $R$, any submonoid $\mathcal{S}\subset R\setminus \{0\}$ consisting of normal elements that aren't zero-divisors is automatically both right and left Ore. In what follows, we will refer to these as \emph{normal} Ore sets. In particular, $\mathcal{S}=\{v_{\lambda}\ |\ \lambda\in \Lambda^+\}\subset \CC_q[G/U]$ is a normal Ore set and the Ore localization $(\CC_q[G/U])[\mathcal{S}^{-1}]$ is isomorphic to $\CC_q[B]$ as $U_q(\gg)$-module algebras. The following lemmas allow us to create normal Ore sets in the $n$-fold braided tensor product $\CC_q[G/U]^{\underline{\otimes}n}$.

\begin{lemma}
\label{lem:ore}
	Let $\kk$ be any field and suppose $A$ and $B$ are $\kk$-algebras such that the $\kk$-vector space $A\otimes_\kk B$ has the structure of a $\kk$-algebra satisfying
$$			(a\otimes 1)(a'\otimes b')=aa'\otimes b',~(a\otimes b)(1\otimes b')=a\otimes bb'
			$$
	for $a,a'\in A$ and $b,b'\in B$. If $\mathcal{S}$ is a normal Ore set in $B$ such that \[(1\otimes s)((A\setminus \{0\})\otimes 1)=((A\setminus\{0\})\otimes 1)(1\otimes s)\] for $s\in \mathcal{S}$, then $1\otimes \mathcal{S}:=\{1\otimes s\ |\ s\in \mathcal{S}\}$ is a normal Ore set in $A\otimes_\kk B$.
\end{lemma}

\begin{proof}
	It is clear that $1\otimes \mathcal{S}$ is a multiplicative set containing $1\otimes 1$ and that $(1\otimes s)(A\otimes_\kk B)=(A\otimes_\kk B)(1\otimes s)$ for $s\in \mathcal{S}$, so we simply show that $1\otimes \mathcal{S}$ does not contain any zero-divisors. Fix $s\in \mathcal{S}$. Now an arbitrary nonzero element $x\in A\otimes_\kk B$ can be written in the form $x=\sum\limits_{k=1}^n a_k\otimes b_k$
	for some $a_k\in A\setminus \{0\}$ and $b_k\in B\setminus\{0\}$. We may assume that $\{b_k\}_{k=1}^n$ is a linearly independent set. Since $s$ is not a zero-divisor in $B$, it follows that $\{sb_k\}_{k=1}^n$ is a linearly independent set, as is $\{b_ks\}_{k=1}^n$. Also, by assumption, for each $k=1,\ldots, n$, there exists $a_k'\in A\setminus\{0\}$ such that $(1\otimes s)(a_k\otimes 1)=a_k'\otimes s$. Then
$$(1\otimes s)x=(1\otimes s)\left(\sum_{k=1}^n a_k\otimes b_k\right)=\sum_{k=1}^n a_k'\otimes sb_k\ne 0,$$
$$x(1\otimes s)=\left(\sum_{k=1}^n a_k\otimes b_k\right)(1\otimes s)=\sum_{k=1}^na_k\otimes b_ks\ne 0.$$
	Since $x$ was an arbitrary element of $A\otimes_\kk B$, we have shown that $1\otimes s$ is not a zero-divisor in $A\otimes_\kk B$.
\end{proof}

\begin{lemma}
\label{lem:braidore}
Let $A$ and $B$ be $U_q(\gg)$-weight module algebras and let $\mathcal{S}$ be a normal Ore set in $B$ consisting of highest weight vectors. Then $1\otimes \mathcal{S}$ is a normal Ore set in the braided tensor product $A\underline{\otimes} B$.
\end{lemma}

\begin{proof}
	In light of Lemma \ref{lem:ore}, it suffices to show that $(1\otimes s)((A\setminus \{0\})\otimes 1)=((A\setminus\{0\})\otimes 1)(1\otimes s)$ for $s\in \mathcal{S}$. Since $\mathcal{S}$ consists of highest weight vectors in $B$, we have the commutation relation 
	\[(1\otimes s)(a\otimes 1)=q^{(|a|,|s|)} a\otimes s\]
	for weight vectors $a\in A$ and $s\in \mathcal{S}$ of weight $|a|$ and $|s|$, respectively. Let us denote $q_{s,a}:=q^{(|a|,|s|)}$. Now an arbitrary nonzero element $a\in A\setminus\{0\}$ is of the form 
		$\sum_{k=1}^n a_k$, 
	where each $a_k\in A\setminus\{0\}$ is a weight vector. We may assume $|a_k|\ne |a_l|$ if $k\ne l$. Then for $s\in \mathcal{S}$,
		\[\sum_{k=1}^n q_{s,a_k}a_k\ne 0\quad \text{ and } \quad \sum_{k=1}^n q_{s,a_k}^{-1}a_k\ne 0.\]
	Therefore since 
		\begin{align*}
			(1\otimes s)\left(\left(\sum_{k=1}^n a_k\right)\otimes 1\right)&=\left(\left(\sum_{k=1}^n q_{s,a_k}a_k\right)\otimes 1\right)(1\otimes s)\quad \text{ and }\\
			\left(\left(\sum_{k=1}^n a_k\right)\otimes 1\right)(1\otimes s)&=(1\otimes s)\left(\left(\sum_{k=1}^n q_{s,a_k}^{-1}a_k\right)\otimes 1\right),
		\end{align*}
	it follows that $(1\otimes s)((A\setminus \{0\})\otimes 1)=((A\setminus\{0\})\otimes 1)(1\otimes s)$ for $s\in \mathcal{S}$ and so the lemma is proven.
\end{proof}

By Lemma \ref{lem:braidore} and induction, $\mathcal{S}':=1\otimes \cdots\otimes 1\otimes \mathcal{S}$ is a normal Ore set in $\CC_q[G/U]^{\underline{\otimes}n}$. Furthermore, it is clear that 
	\[\CC_q[G/U]^{\underline{\otimes}n}[\mathcal{S}'^{-1}]\cong \CC_q[G/U]^{\underline{\otimes}(n-1)}\underline{\otimes} \CC_q[B]\]
as $U_q(\gg)$-module algebras and $\CC_q[G/U]^{\underline{\otimes}(n-1)}\underline{\otimes}\CC_q[B]$ is generated by $(\CC_q[G/U]^{\underline{\otimes}(n-1)}\underline{\otimes}\CC_q[B])^+$ as a $U_q(\gg)$-module algebra. We now have an embedding of $U_q(\gg)$-module algebras 
	\[\CC_q[U]\hookrightarrow \CC_q[B]\subset\CC_q[G/U]^{\underline{\otimes}(n-1)}\underline{\otimes}\CC_q[B].\]
Then by Corollary \ref{cor:qhighest}, $\CC_q[G/U]^{\underline{\otimes}(n-1)}\underline{\otimes}\CC_q[B]\cong (\CC_q[G/U]^{\underline{\otimes} n})[\mathcal{S}'^{-1}]$ is factorizable over $\CC_q[U]$.\qed

\subsection{Proofs of Theorems \ref{thm:qgd} and \ref{thm:qg}}
\label{pf:qgd,qg}

Let $H$ be a Hopf algebra with invertible antipode (e.g. $H=\mathcal{K}$). We will refer to Yetter-Drinfeld modules of various kinds: ${}_H^H\mathcal{YD}$, ${}_H\mathcal{YD}^H$, and $\mathcal{YD}_H^H$ (see, e.g., \cite[Section 2]{cwy}). The side of the subscript denotes the side on which $H$ will act, while the side of the superscript denotes the side on which $H$ will coact. We use sumless Sweedler notation to write left coactions $x\mapsto x^{(-1)}\otimes x^{(0)}$ and right coactions $x\mapsto x^{(0)}\otimes x^{(1)}$. To distinguish the structure maps of a Nichols algebra (a Hopf algebra in the appropriate Yetter-Drinfeld category, see for example \cite{bazlov}) from those of $H$, we underline them. For instance, we write the braided comultiplication $\underline{\Delta}(b)=\underline{b}_{(1)}\otimes\underline{b}_{(2)}$.

We start with some results that will play key roles in the proofs of the Theorems \ref{thm:qgd} and \ref{thm:qg}.

\begin{theorem}
\label{thm:adjoint}
Let $A$ be a left $H$-module algebra and suppose $V\in {}_H^H\mathcal{YD}$ is such that the Nichols algebra $\mathcal{B}(V)$ is a left $H$-module subalgebra of $A$, where ${}_H^H\mathcal{YD}$ is the category of left-left Yetter-Drinfeld modules over $H$. Then $A$ can be given a left $\mathcal{B}(V)$-module structure via
	\[v\rhd a=va-(v^{(-1)}(a))v^{(0)}.\]
\end{theorem}

\begin{proof}
	Consider the Hopf algebra $\widetilde{H}:=\mathcal{B}(V)\rtimes H$, where $\Delta(u)=\underline{u}_{(1)}(\underline{u}_{(2)})^{(-1)}\otimes (\underline{u}_{(2)})^{(0)}$ and \linebreak $S(u)=S(u^{(-1)})\underline{S}(u^{(0)})$ for $u\in \mathcal{B}(V)$. Then $\widetilde{H}$ can naturally be considered as a subalgebra of $\widetilde{A}:=A\rtimes H$. Hence $\widetilde{A}$ is an $\widetilde{H}$-module algebra under the adjoint action:
		\[\widetilde{h}\rhd \widetilde{a}=\widetilde{h}_{(1)}\widetilde{a} S(\widetilde{h}_{(2)}).\]
	We observe that $A$ is preserved under the restriction of this action to $\mathcal{B}(V)$ (note that for convenience we will write, e.g., $a$ instead of $a\otimes 1$):
	\begin{align*}
		u\rhd a&=u_{(1)}aS(u_{(2)})\\
			&=\underline{u}_{(1)}(\underline{u}_{(2)})^{(-1)}aS(((\underline{u}_{(2)})^{(0)})^{(-1)})\underline{S}(((\underline{u}_{(2)})^{(0)})^{(0)})\\
			&=\underline{u}_{(1)}(\underline{u}_{(2)})^{(-2)}aS((\underline{u}_{(2)})^{(-1)})\underline{S}((\underline{u}_{(2)})^{(0)})\\
			&=\underline{u}_{(1)}((\underline{u}_{(2)})^{(-3)}(a))(\underline{u}_{(2)})^{(-2)}S((\underline{u}_{(2)})^{(-1)})\underline{S}((\underline{u}_{(2)})^{(0)})\\
			&=\underline{u}_{(1)}((\underline{u}_{(2)})^{(-2)}(a))\varepsilon((\underline{u}_{(2)})^{(-1)})\underline{S}((\underline{u}_{(2)})^{(0)})\\
			&=\underline{u}_{(1)}((\underline{u}_{(2)})^{(-1)}(a))\underline{S}((\underline{u}_{(2)})^{(0)})\in A
	\end{align*}
	for $u\in \mathcal{B}(V)$ and $a\in A$. In fact, it is clear that $A$ has become a left $\widetilde{H}$-module algebra. Now computing the given action for $v\in V$ and $a\in A$, we find
$$			v\rhd a=\underline{v}_{(1)}((\underline{v}_{(2)})^{(-1)}(a))\underline{S}((\underline{v}_{(2)})^{(0)})
=va+v^{(-1)}(a)\underline{S}(v^{(0)})
=va-v^{(-1)}(a)v^{(0)},$$
	as required. The second and third equalities follow from the fact that every element of $V$ is a primitive element of the braided Hopf algebra $\mathcal{B}(V)$.
\end{proof}

Of course, Theorem \ref{thm:adjoint} has a natural counterpart with ``left" replaced by ``right".

\begin{theorem}
\label{thm:op}
Let $A$ be a right $H$-module algebra and suppose $V\in \mathcal{YD}_H^H$ is such that the Nichols algebra $\mathcal{B}(V)$ is a right $H$-module subalgebra of $A$. Then $A$ can be given a right $\mathcal{B}(V)$-module structure via 
	\[a\blacktriangleleft v= av-v^{(0)}((a) v^{(1)}).\]
\end{theorem}

Given any ring $R$, a right $R$-module is naturally a left $R^{op}$-module, giving us the following obvious corollary.

\begin{corollary}
\label{cor:op}
	In the assumptions of Theorem \ref{thm:op}, if $H$ is commutative, then $A$ can be  given a left $\mathcal{B}(V)^{op}$-module structure via
	\[v\blacktriangleright a=av-v^{(0)}(v^{(1)}(a)).\]
\end{corollary}

\begin{remark}
	If $A$ is a $(\mathcal{B}(V)\rtimes H)$-module algebra (e.g. Theorem \ref{thm:adjoint}), then we can form the \emph{braided cross product} $A\underline{\rtimes}\mathcal{B}(V)$ which, as a vector space, is just $A\otimes \mathcal{B}(V)\subset A\rtimes (\mathcal{B}(V)\rtimes H)$ and it is a subalgebra. Furthermore, it is an $H$-module algebra. We note that if $A$ is additionally a $\mathcal{B}(V)$-module algebra in ${}_H^H\mathcal{YD}$, then our definition of $A\underline{\rtimes}\mathcal{B}(V)$ matches that of $A\rtimes \mathcal{B}(V)$. However, we don't require that $A$ is even an $H$-comodule, which is why we use a different notation. Similarly, we can form the braided tensor product $A\underline{\otimes}\mathcal{B}(V)$ (which is an $H$-module algebra) even if $A$ is an $H$-module algebra and is not in ${}_H^H\mathcal{YD}$, simply satisfying $(1\otimes v)(a\otimes 1)=(v^{(-1)}\rhd a)\otimes v^{(0)}$. This corresponds to the braided cross product $A\underline{\rtimes}\mathcal{B}(V)$, where $\mathcal{B}(V)\rtimes H$ acts on $A$ by the ``trivial" action:
$(u\otimes h)\rhd a=\underline{\varepsilon}(u)h\rhd a\quad \text{for }u\in \mathcal{B}(V),\ h\in H,\ a\in A$.
\end{remark}

\begin{theorem}
\label{thm:braidedtriv}
	Let $V\in {}_H^H\mathcal{YD}$ and suppose $A$ is an $H$-module algebra containing $\mathcal{B}(V)$ as an $H$-module subalgebra. Then the linear map $\tau:=(\mu_A\otimes id)\circ(id\otimes \iota\otimes id)\circ (id\otimes \underline{\Delta}):A\underline{\rtimes}\mathcal{B}(V)\to A\underline{\otimes}\mathcal{B}(V)$ is an $H$-module algebra isomorphism with inverse $\tau^{-1}=(\mu_A\otimes id)\circ(id\otimes \iota\otimes id)\circ (id\otimes \underline{S}\otimes id)\circ (id\otimes \underline{\Delta})$, where $\iota: \mathcal{B}(V)\to A$ is the inclusion and the implied $\mathcal{B}(V)\rtimes H$ action on $A$ is that of Theorem \ref{thm:adjoint}.
\end{theorem}

\begin{proof}
We first verify that $\tau$ and $\tau^{-1}$ are truly mutually inverse (and hence that we are justified in using the name $\tau^{-1}$). For $a\in A$ and $b\in \mathcal{B}(V)$, we directly compute
$$(\tau\circ\tau^{-1})(a\otimes b)=\tau(a\underline{S}(\underline{b}_{(1)})\otimes \underline{b}_{(2)})=aS(\underline{b}_{(1)})\underline{b}_{(2)}\otimes \underline{b}_{(3)}
=a\underline{\varepsilon}(\underline{b}_{(1)})\otimes \underline{b}_{(2)}=a\otimes\underline{\varepsilon}(\underline{b}_{(1)})\underline{b}_{(2)}=a\otimes b,$$
$$(\tau^{-1}\circ\tau)(a\otimes b)=\tau^{-1}(a\underline{b}_{(1)}\otimes \underline{b}_{(2)})
=a\underline{b}_{(1)}\underline{S}(\underline{b}_{(2)})\otimes \underline{b}_{(3)}
=a\underline{\varepsilon}(\underline{b}_{(1)})\otimes \underline{b}_{(2)}
=a\otimes\underline{\varepsilon}(\underline{b}_{(1)})\underline{b}_{(2)}=a\otimes b.
$$	
	Since $\tau\circ \tau^{-1}$ and $\tau^{-1}\circ \tau$ act as the identity on pure tensors, they are both the identity homomorphism. Hence $\tau$ and $\tau^{-1}$ are mutually inverse. We conclude by verifying that $\tau$ is actually a homomorphism of algebras (and hence that $\tau^{-1}$ is as well). For $v\in V$, $a,a'\in A$, and $b\in \mathcal{B}(V)$, we have
	
	\begin{align*}
		\tau(a\otimes v)\tau(a'\otimes b)&=(av\otimes 1+a\otimes v)(a'\underline{b}_{(1)}\otimes \underline{b}_{(2)})\\
			&=ava'\underline{b}_{(1)}\otimes \underline{b}_{(2)}+a(v^{(-1)} (a'\underline{b}_{(1)}))\otimes v^{(0)}\underline{b}_{(2)}\\
			&=ava'\underline{b}_{(1)}\otimes \underline{b}_{(2)}+a(v^{(-2)} (a'))(v^{(-1)}( \underline{b}_{(1)}))\otimes v^{(0)}\underline{b}_{(2)}\\
			&=ava'\underline{b}_{(1)}\otimes \underline{b}_{(2)}-a(v^{(-1)} (a'))v^{(0)}\underline{b}_{(1)}\otimes \underline{b}_{(2)}+a(v^{(-1)}(a'))v^{(0)}\underline{b}_{(1)}\otimes \underline{b}_{(2)}\\
			&\hspace{0.5cm}+a(v^{(-2)}(a'))(v^{(-1)}( \underline{b}_{(1)}))\otimes v^{(0)}\underline{b}_{(2)}\\
			&=\tau(ava'\otimes b-a(v^{(-1)}( a'))v^{(0)}\otimes b+a(v^{(-1)}( a'))\otimes v^{(0)}b)\\
			&=\tau(a(va'-(v^{(-1)}(a'))v^{(0)})\otimes b+a(v^{(-1)}(a'))\otimes v^{(0)}b)\\
			&=\tau(a(v\rhd a')\otimes b+a(v^{(-1)}(a'))\otimes v^{(0)}b)\\
			&=\tau((a\otimes v)(a'\otimes b)).
	\end{align*}

	It is clear that $\{b\in \mathcal{B}(V)\ |\ \tau(a\otimes b)\tau(a'\otimes b')=\tau((a\otimes b)(a'\otimes b'))\ \forall a,a'\in A, b'\in \mathcal{B}(V)\}$ is a subalgebra of $A\underline{\rtimes}\mathcal{B}(V)$. We have shown it contains $V$, so it must be equal to $\mathcal{B}(V)$. Now, since pure tensors span $A\underline{\rtimes} \mathcal{B}(V)$ and $\tau$ is a linear map, it follows that $\tau$ respects multiplication. The theorem is proved.
\end{proof}

\begin{corollary}
\label{cor:embed}
	Let $V\in {}_H^H\mathcal{YD}$. Then there are injective $H$-module algebra homomorphisms \linebreak$\mathcal{B}(V)\to \mathcal{B}(V)\underline{\rtimes}\mathcal{B}(V)$ and $\mathcal{B}(V)\to \mathcal{B}(V)\underline{\otimes}\mathcal{B}(V)$ given by $v\mapsto 1\otimes v-v\otimes 1$  and $v\mapsto 1\otimes v+v\otimes 1$, respectively, for $v\in V$.
\end{corollary}

\begin{proof}
	Let $\tau$ be as in Theorem \ref{thm:braidedtriv}, where $A=\mathcal{B}(V)$. Restrict $\tau^{-1}$ to $\mathcal{B}(V)\cong 1\underline{\otimes} \mathcal{B}(V)\subset \mathcal{B}(V)\underline{\otimes}\mathcal{B}(V)$ and observe that $\tau^{-1}(1\otimes v)=1\otimes v-v\otimes 1$ for $v\in V$.
	
	Similarly, restrict $\tau$ to $\mathcal{B}(V)\cong 1\underline{\rtimes}\mathcal{B}(V)\subset \mathcal{B}(V)\underline{\rtimes}\mathcal{B}(V)$ and note that $\tau(1\otimes v)=1\otimes v+v\otimes 1$ for $v\in V$.
\end{proof}

\begin{theorem}
\label{thm:diff}
	Let $V\in {}_H^H\mathcal{YD}$ and set $\widetilde{H}=\mathcal{B}(V)\rtimes H$. Let $A$ be a left $\widetilde{H}$-module algebra and suppose $A$ contains an $\widetilde{H}$-module subalgebra isomorphic to $\mathcal{B}(V)$ with the ``adjoint" action: 
	\[(u\otimes h)\rhd u'=\underline{u}_{(1)}([(\underline{u}_{(2)})^{(-1)}h](u'))\underline{S}((\underline{u}_{(2)})^{(0)})\quad \text{for } u,u'\in \mathcal{B}(V),\ h\in H.\] 
	Then there is a left $\mathcal{B}(V)$ action $\blacktriangleright$ on $A$ given by
		\[v\blacktriangleright a=(v\rhd a)-[va-(v^{(-1)}( a))v^{(0)}]	\quad \text{for }v\in V,\ a\in A.\]
\end{theorem}

\begin{proof}
	We first observe that $\mathcal{B}(V)\underline{\rtimes}\mathcal{B}(V)$ is an $H$-module subalgebra of $A\underline{\rtimes}\mathcal{B}(V)$. By Corollary \ref{cor:embed}, the elements $1\otimes v-v\otimes 1\in A\underline{\rtimes}\mathcal{B}(V)$ ($v\in V$) generate an $H$-module algebra isomorphic to $\mathcal{B}(V)$. Then by Theorem \ref{thm:adjoint}, we can define an action of $\mathcal{B}(V)$ on $A\underline{\rtimes}\mathcal{B}(V)$ by 
		\[v\cdot (a\otimes u)=(1\otimes v-v\otimes 1)(a\otimes u)-[v^{(-1)}(a\otimes u)](1\otimes v^{(0)}-v^{(0)}\otimes 1).\]

Now we need only observe that this action preserves $A=A\underline{\rtimes} 1\subset A\underline{\rtimes}\mathcal{B}(V)$ and acts in the prescribed manner:
	\begin{align*}
		v\cdot (a\otimes 1)&=(1\otimes v-v\otimes 1)(a\otimes 1)-[v^{(-1)}(a\otimes 1)](1\otimes v^{(0)}-v^{(0)}\otimes 1)\\
			&=(v\rhd a)\otimes 1+(v^{(-1)}( a))\otimes v^{(0)}-va\otimes 1-(v^{(-1)}( a))\otimes v^{(0)}+(v^{(-1)}( a))v^{(0)}\otimes 1\\
			&=[(v\rhd a)-[va-(v^{(-1)}( a))v^{(0)}]]\otimes 1.
	\end{align*}
\end{proof}

\begin{theorem}
\label{thm:sum}
	Let $V\in {}_H^H\mathcal{YD}$ and set $\widetilde{H}=\mathcal{B}(V)\rtimes H$. Let $A$ be a left $\widetilde{H}$-module algebra and suppose $A$ contains an $\widetilde{H}$-module subalgebra isomorphic to $\mathcal{B}(V)$ with the trivial action:
	\[(u\otimes h)\rhd u'=[\underline{\varepsilon}(u)h]( u')\quad \text{for } u,u'\in \mathcal{B}(V),\ h\in H.\] 
	Then there is a left $\mathcal{B}(V)$ action $\blacktriangleright$ on $A$ given by 
		\[v\blacktriangleright a=(v\rhd a)+[va-(v^{(-1)}( a))v^{(0)}]
	\quad \text{for }v\in V,\ a\in A.\]
\end{theorem}

\begin{proof}
	By Corollary \ref{cor:embed}, the elements $1\otimes v+v\otimes 1\in A\underline{\rtimes}\mathcal{B}(V)$ ($v\in V$) generate an $H$-module algebra isomorphic to $\mathcal{B}(V)$. Then by Theorem \ref{thm:adjoint}, we can define an action of $\mathcal{B}(V)$ on $A\underline{\rtimes}\mathcal{B}(V)$ by 
		\[v\cdot (a\otimes u)=(1\otimes v+v\otimes 1)(a\otimes u)-[v^{(-1)} (a\otimes u)](1\otimes v^{(0)}+v^{(0)}\otimes 1).\]
	Now we need only observe that this action preserves $A=A\underline{\rtimes}1\subset A\underline{\rtimes}\mathcal{B}(V)$ and acts in the prescribed manner:
		\begin{align*}
			v\cdot (a\otimes 1)&=(1\otimes v+v\otimes 1)(a\otimes 1)-[v^{(-1)} (a\otimes 1)](1\otimes v^{(0)}+v^{(0)}\otimes 1)\\
				&=(v\rhd a)\otimes 1+(v^{(-1)}( a))\otimes v^{(0)}+va\otimes 1-(v^{(-1)}( a))\otimes v^{(0)}-(v^{(-1)}( a))v^{(0)}\otimes 1\\
				&=[(v\rhd a)+[va-(v^{(-1)}( a))v^{(0)}]]\otimes 1
		\end{align*}
\end{proof}

\begin{theorem}
\label{thm:commute}
	Let $\hat{A}$ be a left $H$-module algebra and suppose that $V_1\in {}_H^H\mathcal{YD}$ and $V_2\in {}_H\mathcal{YD}^H$ are $H$-submodules of $\hat{A}$. For $v_i\in V_i$, define the following actions on $\hat{A}$:
	\[v_1\rhd a=v_1a-v_1^{(-1)}(a)v_1^{(0)}\quad v_2\blacktriangleright a=v_2^{(0)}v_2^{(1)}(a)-av_2.\]
	If (1) $v_1\rhd v_2=v_2\blacktriangleright v_1=0$ and (2) $v_1^{(-2)}(v_2^{(0)})v_1^{(-1)}(v_2^{(1)}(a))v_1^{(0)}=v_2^{(0)}v_2^{(1)}(v_1^{(-1)}(a))v_2^{(2)}(v_1^{(0)})$ for all $v_i\in V_i$ and $a\in \hat{A}$, then $v_1\rhd (v_2\blacktriangleright a)=v_2\blacktriangleright(v_1\rhd a)$ for $v_i\in V_i$ and $a\in \hat{A}$.
\end{theorem}

\begin{proof}
	We first note that if $v_1\rhd v_2=v_2\blacktriangleright v_1=0$, then $v_1v_2=v_1^{(-1)}(v_2)v_1^{(0)}=v_2^{(0)}v_2^{(1)}(v_1)$ and for $a,b\in A$, we have
$v_1\rhd(ab)=(v_1\rhd a)b+v_1^{(-1)}(a)(v_1^{(0)}\rhd b); \quad v_2\blacktriangleright(ab)=(v_2^{(0)}\blacktriangleright a)v_2^{(1)}(b)+a(v_2\triangleright b)$.
	Now we simply compute:
		\begin{align*}
			v_1\rhd (v_2\blacktriangleright a)&=v_1\rhd(v_2^{(0)}v_2^{(1)}(a)-av_2)\\
				&=(v_1\rhd v_2^{(0)})v_2^{(1)}(a)+v_1^{(-1)}(v_2^{(0)})(v_1^{(0)}\rhd v_2^{(1)}(a))-(v_1\rhd a)v_2-v_1^{(-1)}(a)(v_1^{(0)}\rhd v_2)\\
				&=v_1^{(-1)}(v_2^{(0)})(v_1^{(0)}\rhd v_2^{(1)}(a))-(v_1\rhd a)v_2\\
				&=v_1^{(-1)}(v_2^{(0)})(v_1^{(0)}v_2^{(1)}(a)-(v_1^{(0)})^{(-1)}(v_2^{(1)}(a))(v_1^{(0)})^{(0)})-(v_1a-v_1^{(-1)}(a)v_1^{(0)})v_2\\
				&=v_1^{(-1)}(v_2^{(0)})v_1^{(0)}v_2^{(1)}(a)-v_1^{(-2)}(v_2^{(0)})v_1^{(-1)}(v_2^{(1)}(a))v_1^{(0)}-v_1av_2+v_1^{(-1)}(a)v_1^{(0)}v_2\\\\
			v_2\blacktriangleright(v_1\rhd a)&=v_2\blacktriangleright(v_1a-v_1^{(-1)}(a)v_1^{(0)})\\
				&=(v_2^{(0)}\blacktriangleright v_1)v_2^{(1)}(a)+v_1(v_2\blacktriangleright a)-(v_2^{(0)}\blacktriangleright v_1^{(-1)}(a))v_2^{(1)}(v_1^{(0)}+v_1^{(-1)}(a)(v_2\triangleright v_1^{(0)}\\
				&=v_1(v_2\blacktriangleright a)-(v_2^{(0)}\blacktriangleright v_1^{(-1)}(a))v_2^{(1)}(v_1^{(0)})\\
				&=v_1(v_2^{(0)}v_2^{(1)}(a)-av_2)-((v_2^{(0)})^{(0)}(v_2^{(0)})^{(1)}(v_1^{(-1)}(a))-v_1^{(-1)}(a)v_2^{(0)})v_2^{(1)}(v_1^{(0)})\\
				&=v_1v_2^{(0)}v_2^{(1)}(a)-v_1av_2-v_2^{(0)}v_2^{(1)}(v_1^{(-1)}(a))v_2^{(2)}(v_1^{(0)})+v_1^{(-1)}(a)v_2^{(0)}v_2^{(1)}(v_1^{(0)})
		\end{align*}
	Comparing terms, we see that the two quantities are indeed equal.
\end{proof}

We are now ready to prove Theorem \ref{thm:qgd}.

\noindent {\bf Proof of Theorem \ref{thm:qgd}}.
Let $V=\text{span}_{\CC(q)}\{F_i\ |\ i\in I\}\subset U_q(\bb_-)$. Then $V\in {}_\mathcal{K}^\mathcal{K}\mathcal{YD}$ with structure given by 
	\[K_i^{\pm1}\rhd F_j=q_i^{\mp c_{i,j}}F_j;\quad \delta_L(F_i)=K_i^{-1}\otimes F_i.\]
Let $V'=V$ as a vector space, but $V'\in {}_\mathcal{K}\mathcal{YD}^\mathcal{K}$ with structure given by
	\[K_i^{\pm1}\rhd F_j=q_i^{\mp c_{i,j}}F_j;\quad \delta_R(F_i)=F_i\otimes K_i^{-1}.\]
Note that we can also consider $V'$ as an object of $\mathcal{YD}_\mathcal{K}^\mathcal{K}$ since $\mathcal{K}$ is commutative. It is well-known (see, e.g., \cite{ahs} or \cite{lusztig}, though Lusztig never used the term ``Nichols algebra") that the corresponding Nichols algebras are isomorphic to $U_q(\nn_-)$ as $\mathcal{K}$-module algebras in the obvious way, i.e. $F_i\mapsto F_i$. 

By assumption, $\CC_q[U]\subset A$, so there is a natural embedding of $U_q(\bb_-)$-module algebras $U_q(\nn_-)\hookrightarrow A$ given by $F_i\mapsto \frac{x_i}{q_i-q_i^{-1}}$.  Theorems \ref{thm:adjoint} and \ref{thm:diff} then imply that there is a $U_q(\nn_-)$ action on $A$ given by 
	\[F_i\rhd a=F_i(a)-\frac{x_ia-K_i^{-1}(a)x_i}{q_i-q_i^{-1}},\]
matching the proposed action of $F_{i,1}$.

Now utilizing a slightly different embedding $U_q(\nn_-)\hookrightarrow A,\ F_i\mapsto -\frac{x_i}{q_i-q_i^{-1}}$, Corollary \ref{cor:op} gives another action of $U_q(\nn_-)\cong U_q(\nn_-)^{op}$ on $A$:
	\[F_i\rhd a=\frac{x_iK_i^{-1}(a)-ax_i}{q_i-q_i^{-1}},\]
matching the proposed action of $F_{i,2}$.

It is easily observed that we have made $A$ into both a $\mathcal{B}(V)\rtimes \mathcal{K}$-module algebra and a $\mathcal{B}(V')\rtimes \mathcal{K}$-module algebra.

We now wish to show that the operators $F_{i,1}$ and $F_{j,2}$ commute. To do so, we construct the braided cross product $\hat{A}:=A\underline{\rtimes} U_q(\nn_-)$, where the $F_i$ act as $F_{i,1}$. As above, we define ``clever" embeddings of $V$ and $V'$ into $\hat{A}$, namely $F_i\mapsto \frac{x_i}{q_i-q_i^{-1}}\otimes1$ and $F_i\mapsto 1\otimes F_i$, respectively. It is easily checked that the hypotheses of Theorem \ref{thm:commute} are satisfied. Furthermore, the actions defined in Theorem \ref{thm:commute} preserve $A$ and match the actions of $F_{i,1}$ and $F_{j,2}$ on $A$, showing that the prescribed actions of $F_{i,1}$ and $F_{j,2}$ do, in fact, commute.\qed

In light of Theorem \ref{thm:qgd}, $\CC_q[U]$ is a $U_q(\gg^*)$-module algebra with action given by
	\[K_i^{\pm1}\rhd x_j=q_i^{\mp c_{i,j}}x_j; \quad F_{i,1}\rhd x_j=0;\quad  F_{i,2}\rhd x_j=\frac{q_i^{c_{i,j}}x_ix_j-x_jx_i}{q_i-q_i^{-1}}.\]
	
	We make $\CC_q[U]$ into a $U_q(\gg^*)$-comodule algebra via the algebra homomorphism $\delta:\CC_q[U]\to U_q(\gg^*)\otimes \CC_q[U]$ given on generators by
	\[\delta(x_i)=K_i\otimes x_i+(q_i-q_i^{-1})F_{i,2}K_i\otimes 1.\]
The fact that this gives a well-defined algebra homomorphism follows immediately from the following lemma, which can be deduced from the fact that in \cite[1.2.6]{lusztig}, $r:{\bf f}\to {\bf f}\otimes {\bf f}$, $\theta_i\mapsto \theta_i\otimes 1+1\otimes \theta_i$ is well-defined.

\begin{lemma}
\label{lem:qsr}
	Let $R$ be any $\CC(q)$-algebra and suppose $\{y_i\}_{i\in I},\{z_i\}_{i\in I}\subseteq R$ are two families of elements satisfying the quantum Serre relations. If
		\[z_jy_i=q_i^{c_{i,j}}y_iz_j\quad \text{for }i,j\in I\]
	then $\{y_i+z_i\}_{i\in I}$ also satisfies the quantum Serre relations.
\end{lemma}

It is easily checked that $(\text{id}\otimes \delta)\circ\delta=(\Delta\otimes \text{id})\circ \delta$ and $(\epsilon\otimes \text{id})\circ \delta=\text{id}$. 

\begin{proposition}
\label{prop:ydCqU}
	The above action and coaction make $\CC_q[U]$ into an algebra in the category ${}_{U_q(\gg^*)}^{U_q(\gg^*)}\mathcal{YD}$ of left-left Yetter-Drinfeld modules over $U_q(\gg^*)$.
\end{proposition}

\begin{proof}
	We need only verify that the compatibility condition is satisfied, i.e. that
		\begin{equation} 
		\label{eq:ydcomp}
			h_{(1)}x^{(-1)}\otimes (h_{(2)}\rhd x^{(0)})=(h_{(1)}\rhd x)^{(-1)}h_{(2)}\otimes (h_{(1)}\rhd x)^{(0)}
		\end{equation}
	for $h\in U_q(\gg^*)$ and $x\in \CC_q[U]$. It is easily checked that \eqref{eq:ydcomp} is satisfied for $h\in \{K_i^{\pm1},F_{i,1},F_{i,2}\ |\ i\in I\}$ and $x\in \{x_i\ |\ i\in I\}$. Suppose \eqref{eq:ydcomp} is satisfied for some $x,x'\in \CC_q[U]$ and all $h\in \{K_i^{\pm1},F_{i,1},F_{i,2}\ |\ i\in I\}$. Then since $\Delta(\{K_i^{\pm1},F_{i,1},F_{i,2}\ |\ i\in I\})\subset \{K_i^{\pm1},F_{i,1},F_{i,2}\ |\ i\in I\}\otimes \{K_i^{\pm1},F_{i,1},F_{i,2}\ |\ i\in I\}$, we observe:
		
	\begin{align*}
		h_{(1)}(xx')^{(-1)}\otimes (h_{(2)}\rhd (xx')^{(0)})&=h_{(1)}x^{(-1)}(x')^{(-1)}\otimes (h_{(2)}\rhd (x^{(0)}(x')^{(0)}))\\
		&=h_{(1)}x^{(-1)}(x')^{(-1)}\otimes (h_{(2)}\rhd x^{(0)})(h_{(3)}\rhd (x')^{(0)})\\
		&=(h_{(1)}\rhd x)^{(-1)}h_{(2)}(x')^{(-1)}\otimes (h_{(1)}\rhd x)^{(0)}(h_{(3)}\rhd (x')^{(0)})\\
		&=(h_{(1)}\rhd x)^{(-1)}(h_{(2)}\rhd x')^{(-1)}h_{(3)}\otimes (h_{(1)}\rhd x)^{(0)}(h_{(2)}\rhd x')^{(0)}\\
		&=((h_{(1)}\rhd x)(h_{(2)}\rhd x'))^{(-1)}h_{(3)}\otimes ((h_{(1)}\rhd x)(h_{(2)}\rhd x'))^{(0)}\\
		&=(h_{(1)}\rhd (xx'))^{(-1)}h_{(2)}\otimes (h_{(1)}\rhd (xx'))^{(0)}
	\end{align*}
for $h\in \{K_i^{\pm1},F_{i,1},F_{i,2}\ |\ i\in I\}$. Hence we see that the set of all $x\in \CC_q[U]$ such that \eqref{eq:ydcomp} holds for all $h\in \{K_i^{\pm1},F_{i,1},F_{i,2}\ |\ i\in I\}$ is a subalgebra of $\CC_q[U]$ containing $\{x_i\ |\ i\in I\}$. Namely, \eqref{eq:ydcomp} holds for all $x\in \CC_q[U]$ and $h\in \{K_i^{\pm1},F_{i,1},F_{i,2}\ |\ i\in I\}$. Now suppose \eqref{eq:ydcomp} holds for some $h,h'\in U_q(\gg^*)$ and all $x\in \CC_q[U]$. Then we observe:
	
	\begin{align*}
		(hh')_{(1)}x^{(-1)}\otimes ((hh')_{(2)}\rhd x^{(0)})&=h_{(1)}h_{(1)}'x^{(-1)}\otimes ((h_{(2)}h_{(2)}')\rhd x^{(0)})\\
		&=h_{(1)}h_{(1)}'x^{(-1)}\otimes (h_{(2)}\rhd(h_{(2)}'\rhd x^{(0)}))\\
		&=h_{(1)}(h_{(1)}'\rhd x)^{(-1)}h_{(2)}'\otimes (h_{(2)}\rhd (h_{(1)}'\rhd x)^{(0)})\\
		&=(h_{(1)}\rhd (h_{(1)}'\rhd x))^{(-1)}h_{(2)}h_{(2)}'\otimes (h_{(1)}\rhd (h_{(1)}'\rhd x))^{(0)}\\
		&=((h_{(1)}h_{(1)}')\rhd x)^{(-1)}h_{(2)}h_{(2)}'\otimes ((h_{(1)}h_{(1)}')\rhd x)^{(0)}\\
		&=((hh')_{(1)}\rhd x)^{(-1)}(hh')_{(2)}\otimes ((hh')_{(1)}\rhd x)^{(0)}
	\end{align*}
for $x\in \CC_q[U]$. Hence we see that the set of all $h\in U_q(\gg^*)$ such that \eqref{eq:ydcomp} holds for all $x\in \CC_q[U]$ is a subalgebra of $U_q(\gg^*)$ containing $\{K_i^{\pm1},F_{i,1},F_{i,2}\ |\ i\in I\}$. Namely, \eqref{eq:ydcomp} holds for all $x\in \CC_q[U]$ and $h\in U_q(\gg^*)$.
\end{proof}


The following proposition is probably well-known, but a source was not quickly found, so we provide a proof here.

\begin{proposition}
\label{prop:ydtensor}
	Let $H$ be a $\kk$-bialgebra, $A$ an $H$-module algebra, and $B$ an algebra in ${}_H^H\mathcal{YD}$. Then the $H$-module $A\otimes_\kk B$ is an $H$-module algebra with multiplication given by
	\[(a\otimes b)(a'\otimes b')=a(b^{(-1)}\rhd a')\otimes b^{(0)}b',\quad \text{for all }a,a'\in A,\ b,b'\in B,\]
where $\rhd$ is the action of $H$ and $\delta(b)=b^{(-1)}\otimes b^{(0)}$ is the coaction of $H$ in sumless Sweedler notation.
\end{proposition}

\begin{proof}
	We first show that $A\otimes_\kk B$ is indeed an associative algebra under the prescribed multiplication. For $a,a',a''\in A$ and $b,b',b''\in B$, we have
		\begin{align*}
			((a\otimes b)(a'\otimes b'))(a''\otimes b'')&=(a(b^{(-1)}\rhd a')\otimes b^{(0)}b)(a''\otimes b'')\\
			&=a(b^{(-1)}\rhd a')((b^{(0)}b')^{(-1)}\rhd a'')\otimes (b^{(0)}b)^{(0)}b''\\
			&=a(b^{(-1)}\rhd a')((b^{(0)})^{(-1)}(b')^{(-1)}\rhd a'')\otimes (b^{(0)})^{(0)}(b')^{(0)}b''\\
			&=a(b^{(-2)}\rhd a')(b^{(-1)}(b')^{(-1)}\rhd a'')\otimes b^{(0)}(b')^{(0)}b''\\
			&=a(b^{(-1)}\rhd (a'((b')^{(-1)}\rhd a'')))\otimes b^{(0)}(b')^{(0)}b''\\
			&=(a\otimes b)(a'((b')^{(-1)}\rhd a'')\otimes (b')^{(0)}b'')\\
			&=(a\otimes b)((a'\otimes b')(a''\otimes b'')).
		\end{align*}
	Hence the prescribed multiplication is associative. We now verify that $A\otimes_\kk B$ is indeed an $H$-module algebra. For $h\in H$, $a,a'\in A$, and $b,b'\in B$, we have
		\begin{align*}
			h\rhd ((a\otimes b)(a'\otimes b'))&=h\rhd(a(b^{(-1)}\rhd a')\otimes b^{(0)}b')\\
			&=(h_{(1)}\rhd a)(h_{(2)}b^{(-1)}\rhd a')\otimes (h_{(3)}\rhd b^{(0)})(h_{(4)}\rhd b')\\
			&=(h_{(1)}\rhd a)((h_{(2)}\rhd b)^{(-1)}h_{(3)}\rhd a')\otimes (h_{(2)}\rhd b)^{(0)}(h_{(4)}\rhd b')\\
			&=(h_{(1)}\rhd a)((h_{(2)}\rhd b)^{(-1)}\rhd (h_{(3)}\rhd a'))\otimes (h_{(2)}\rhd b)^{(0)}(h_{(4)}\rhd b')\\
			&=((h_{(1)}\rhd a)\otimes (h_{(2)}\rhd b))((h_{(3)}\rhd a')\otimes(h_{(4)}\rhd b'))\\
			&=(h_{(1)}\rhd (a\otimes b))(h_{(2)}\rhd (a'\otimes b')).
		\end{align*}
$$			h\rhd (1\otimes 1)=(h_{(1)}\otimes 1)\otimes (h_{(2)}\rhd 1)
=\varepsilon(h_{(1)})\otimes \varepsilon(h_{(2)})
=\varepsilon(h_{(1)}\varepsilon(h_{(2)}))\otimes 1 
=\varepsilon(h)\otimes 1.$$
	The proposition is proven.
\end{proof}

\noindent {\bf Proof of Theorem \ref{thm:qg}}. By Propositions \ref{prop:ydCqU} and \ref{prop:ydtensor} we may give $A\otimes \CC_q[U]$ a $U_q(\gg^*)$-module algebra structure satisfying
	\begin{align*}
		(1\otimes x_i)(a\otimes 1)&=K_i(a)\otimes x_i+(q_i-q_i^{-1})F_{i,2}K_i(a)\otimes 1\\
		K_{i,1}^{\pm1}\rhd (a\otimes x)&=K_i^{\pm1}(a)\otimes K_i^{\pm1}(x)\\
		F_{i,1}\rhd (a\otimes x)&=F_{i,1}(a)\otimes x+K_i^{-1}(a)\otimes F_{i,1}(x)\\
		F_{i,2}\rhd (a\otimes x)&=F_{i,2}(a)\otimes K_i^{-1}(x)+a\otimes F_{i,2}(x)
	\end{align*}
for $i\in I$, $a\in A$, and $x\in \CC_q[U]$.

Now by Theorem \ref{thm:sum}, there is a left action of $U_q(\nn_-)$ on $A\otimes \CC_q[U]$ given by 
	\begin{align*}
		F_i\rhd(a\otimes x)&=F_{i,1}\rhd(a\otimes x)+\frac{(1\otimes x_i)(a\otimes x)-(K_i^{-1}\rhd (a\otimes x))(1\otimes x_i)}{q_i-q_i^{-1}}\\
			&=F_{i,1}(a)\otimes x+\frac{K_i(a)\otimes x_ix+(q_i-q_i^{-1})F_{i,2}K_i(a)\otimes x-K_i^{-1}(a)\otimes K_i^{-1}(x)x_i}{q_i-q_i^{-1}}\\
			&=(F_{i,1}(a)+F_{i,2}K_i(a))\otimes x+\frac{K_i(a)\otimes x_ix-K_i^{-1}(a)\otimes K_i^{-1}(x)x_i}{q_i-q_i^{-1}}\\
			&=(F_{i,1}(a)+F_{i,2}K_i(a))\otimes x+\frac{K_i(a)\otimes x_ix-K_i^{-1}(a)\otimes x_ix+K_i^{-1}(a)\otimes x_ix-K_i^{-1}(a)\otimes K_i^{-1}(x)x_i}{q_i-q_i^{-1}}\\
			&=(F_{i,1}(a)+F_{i,2}K_i(a))\otimes x+\frac{K_i(a)-K_i^{-1}(a)}{q_i-q_i^{-1}}\otimes x_ix+K_i^{-1}(a)\otimes F_i(x),
	\end{align*}
matching the proposed action of $F_i$.

Furthermore, it is obvious that $E_i\rhd (a\otimes x)=a\otimes E_i(x)$ yields a well-defined action of $U_q(\nn_+)$ on $A\otimes \CC_q[U]$. It is now straight-forward to check that
		\begin{align*}
			K_i^{\pm1}\triangleright (K_j^{\pm1}\triangleright (a\otimes x))&=K_j^{\pm1}\triangleright (K_i^{\pm 1}\triangleright (a\otimes x))\\
			K_i\triangleright(E_{j}\triangleright(K_i^{-1}((a\otimes x))))&=q_i^{c_{i,j}}E_{j}\triangleright (a\otimes x)\\
			K_i\triangleright(F_{j}\triangleright(K_i^{-1}((a\otimes x))))&=q_i^{-c_{i,j}}F_{j}\triangleright (a\otimes x)\\
			E_i\triangleright (F_j\triangleright (a\otimes x))-F_j\triangleright (E_i\triangleright (a\otimes x))&=\delta_{i,j}\frac{K_i\triangleright (a\otimes x)-K_i^{-1}\triangleright (a\otimes x)}{q_i-q_i^{-1}}
		\end{align*}
		\begin{align*}
			K_i^{\pm1}(1\otimes 1)&=1\otimes 1\\
			E_i(1\otimes 1)&=0\\
			F_i(1\otimes 1)&=0.
		\end{align*}
	
Hence we have given $A\otimes \CC_q[U]$ the structure of a $U_q(\gg)$-module. To see that it is in fact a module algebra, we need to check the following.
		\begin{align}
		\label{eq:Kcomp}	K_i^{\pm1}\triangleright ((a\otimes x)(a'\otimes x'))&=(K_i^{\pm1}\triangleright (a\otimes x))(K_i^{\pm 1}\triangleright (a'\otimes x'))\\
		\label{eq:Ecomp}	E_{i}\triangleright ((a\otimes x)(a'\otimes x'))&=(E_{i}\triangleright (a\otimes x))(K_i\triangleright (a'\otimes x'))+(a\otimes x)(E_{i}\triangleright (a'\otimes x'))\\
		\label{eq:Fcomp}	F_{i}\triangleright ((a\otimes x)(a'\otimes x'))&=(F_{i}\triangleright (a\otimes x))(a'\otimes x')+(K_i^{-1}\triangleright(a\otimes x))(F_{i}\triangleright (a'\otimes x')).
		\end{align}
		
	Rather than direct verification, we begin by observing that 
$$			h\triangleright ((a\otimes 1)z)=(h_{(1)}\triangleright (a\otimes 1))(h_{(2)}\triangleright z)~\text{ and}~
			h\triangleright ((a\otimes x_j)z)=(h_{(1)}\triangleright (a\otimes x_j))(h_{(2)}\triangleright z)$$
		for $h\in \{K_i^{\pm 1}, E_i, F_i\ |\ i\in I\},\ j\in I,\ a\in A$, and $z\in A\otimes \CC_q[U]$. Let \[Y=\{x\in \CC_q[U]\ |\ h\triangleright ((a\otimes x)z)=(h_{(1)}\triangleright (a\otimes x))(h_{(2)}\triangleright z) \text{ for all }a\in A,\ z\in A\otimes \CC_q[U]\text{ and }h\in \{K_i^{\pm1},E_i,F_i\ |\ i\in I\}\}.\] Then $Y$ is clearly a $\CC(q)$-vector space (containing 1 and $x_j$). We show that $Y$ is closed under multiplication. Suppose $x,x'\in Y,$ $a\in A,$ and $z\in A\otimes \CC_q[U]$. Then for $h\in \{K_i^{\pm 1},E_i,F_i\ |\ i\in I\}$,
		\begin{align*}
			h\triangleright ((a\otimes xx')z)&=h\triangleright ((a\otimes x)(1\otimes x')z)\\
				&=(h_{(1)}\triangleright (a\otimes x))(h_{(2)}\triangleright ((1\otimes x')z))\\
				&=(h_{(1)}\triangleright (a\otimes x))(h_{(2)}\triangleright (1\otimes x'))(h_{(3)}(z))\\
				&=(h_{(1)}\triangleright (a\otimes x)(1\otimes x'))(h_{(2)}(z))\\
				&=(h_{(1)}\triangleright (a\otimes xx'))(h_{(2)}(z)).
		\end{align*}
		Hence $xx'\in Y$ and we have shown that $Y$ is closed under multiplication. It follows that $Y$ is a $\CC(q)$-subalgebra of $\CC_q[U]$ containing $x_j$ and hence is actually $\CC_q[U]$ itself. Hence we have verified equations \eqref{eq:Kcomp}, \eqref{eq:Ecomp}, and \eqref{eq:Fcomp}. It follows that the given structure makes $A\otimes \CC_q[U]$ into a $U_q(\gg)$-module algebra. \qed

\subsection{Proof of Theorem \ref{thm:qequiv}}
\label{pf:qequiv}

We begin by constructing a natural isomorphism $\psi: (-)^+\otimes\CC_q[U]\Rightarrow \text{id}_{\mathcal{C}_\gg^q}$.	For every object $(A,\varphi_A)$ of $\mathcal{C}_\gg^q$, set $\psi_{(A,\varphi_A)}:=m_A\circ (\iota_A\otimes \varphi_A)$, where $\iota_A$ is the inclusion $A^+\hookrightarrow A$ and $m_A:A\otimes A\to A$ is multiplication. As an abuse of notation, we will write $\psi_A$ when context is clear. Since $\psi_A$ is clearly a linear map, we check that it respects multiplication and is $U_q(\gg)$-equivariant. One easily computes
	\[\psi_A((a\otimes 1)(a'\otimes x'))=\psi_A(a\otimes 1)\psi_A(a'\otimes x')\quad \text{and} \quad \psi_A((a\otimes x_i)(a'\otimes x'))=\psi_A(a\otimes x_i)\psi_A(a'\otimes x').\]
	Let $Y=\{x\in \CC_q[U]\ |\ \psi_A((a\otimes x)z)=\psi_A(a\otimes x)\psi_A(z)\ \forall a\in A, z\in A\otimes \CC_q[U]\}$. We have seen that $1,x_i\in Y$ for $i\in I$, so the computations
	\begin{align*}
		\psi_A((a\otimes (x+y))z)&=\psi_A((a\otimes x)z+(a\otimes y)z)\\
			&=\psi_A((a\otimes x)z)+\psi_A((a\otimes y)z)\\
			&=\psi_A(a\otimes x)\psi_A(z)+\psi_A(a\otimes y)\psi_A(z)\\
			&=(\psi_A(a\otimes x)+\psi_A(a\otimes y))\psi_A(z)\\
			&=\psi_A(a\otimes (x+y))\psi_A(z)
	\end{align*}
	\begin{align*}
		\psi_A((a\otimes xy)z)&=\psi_A((a\otimes x)(1\otimes y)z)\\
			&=\psi_A(a\otimes x)\psi_A((1\otimes y)z)\\
			&=\psi_A(a\otimes x)\psi_A(1\otimes y)\psi_A(z)\\
			&=\psi_A((a\otimes x)(1\otimes y))\psi_A(z)\\
			&=\psi_A((a\otimes xy))\psi_A(z)
	\end{align*}
show that $Y$ is a subalgebra of $\CC_q[U]$ containing a generating set. Hence $Y=\CC_q[U]$, i.e. $\psi_A$ is a homomorphism of algebras. Now we verify that $\psi_A$ is $U_q(\gg)$-invariant. For $i\in I$, we have
		\begin{align*}
			K_i^{\pm1}(\psi_A(a\otimes x))&=K_i^{\pm1}(a\varphi_A(x))\\
				&=K_i^{\pm1}(a)K_i^{\pm1}(\varphi_A(x))\\
				&=K_i^{\pm1}(a)\varphi_A(K_i^{\pm1}(x))\\
				&=\psi_A(K_i^{\pm1}(a)\otimes K_i^{\pm1}(x))\\
				&=\psi_A(K_i^{\pm1}\rhd (a\otimes x))
		\end{align*}
		\begin{align*}
			E_i(\psi_A(a\otimes x))&=E_i(a\varphi_A(x))\\
							&=E_i(a)K_i(\varphi_A(x))+aE_i(\varphi_A(x))\\
							&=a\varphi_A(E_i(x))\\
							&=\psi_A(a\otimes E_i(x))\\
							&=\psi_A(E_i\rhd(a\otimes x))
		\end{align*}
		\begin{align*}
			F_i(\psi_A(a\otimes x))&=F_i(a\varphi_A(x))\\
							&=F_i(a)\varphi_A(x)+K_i^{-1}(a)F_i(\varphi_A(x))\\
							&=\left(F_i(a)-\frac{\varphi_A(x_i)a-K_i^{-1}(a)\varphi_A(x_i)}{q_i-q_i^{-1}}\right)\varphi_A(x)+\frac{\varphi_A(x_i)a-K_i(a)\varphi_A(x_i)}{q_i-q_i^{-1}}\varphi_A(x)\\
							&\hspace{0.5cm}+\frac{K_i(a)-K_i^{-1}(a)}{q_i-q_i^{-1}}\varphi_A(x_i)\varphi_A(x)+K_i^{-1}(a)\varphi_A(F_i(x))\\
							&=\left(F_{i,1}(a)+F_{i,2}K_i(a)\right)\varphi_A(x)+\frac{K_i(a)-K_i^{-1}(a)}{q_i-q_i^{-1}}\varphi(x_ix)+K_i^{-1}(a)\varphi_A(F_i(x))\\
							&=\psi_A\left(\left(F_{i,1}(a)+F_{i,2}K_i(a)\right)\otimes x+\frac{K_i(a)-K_i^{-1}(a)}{q_i-q_i^{-1}}\otimes x_ix+K_i^{-1}(a)\otimes F_i(x)\right)\\
							&=\psi_A(F_i\rhd (a\otimes x))
		\end{align*}
	So $\psi_A$ is a homomorphism of $U_q(\gg)$-modules and thus a homomorphism of $U_q(\gg)$-module algebras.  By Theorem \ref{thm:qfactoring}, $\psi_A$ is an isomorphism of $U_q(\gg)$-module algebras. Now 
$$			\psi_A\circ (1\otimes \text{id})=m_A\circ (\iota_A\otimes \varphi_A)\circ(1\otimes \text{id})=m_A\circ(1\otimes \varphi_A)
=\varphi_A.$$
	Hence $\psi_A$ is a morphism of $\mathcal{C}_\gg^q$. To show that $\psi_A$ is an isomorphism in $\mathcal{C}_\gg^q$, we make the following easy observation.
	
\begin{lemma}
\label{lem:qA-iso}
	A morphism between objects of $\mathcal{C}_\gg^q$ is an isomorphism if and only if the underlying homomorphism of $U_q(\gg)$-module algebras is an isomorphism.
\end{lemma}

\begin{proof}
	It is clear that the homomorphism of $U_q(\gg)$-module algebras which underlies an isomorphism between objects of $\mathcal{C}_\gg^q$ is actually an isomorphism, so we simply show the converse. Let $(A,\varphi_A)$ and $(B,\varphi_B)$ be objects of $\mathcal{C}_\gg^q$ and $\xi:A\to B$ a morphism between them such that $\xi$ is an isomorphism of $U_q(\gg)$-module algebras. Then $\xi\circ \varphi_A=\varphi_B$. Hence we have $\xi^{-1}\circ \varphi_B=\xi^{-1}\circ\xi\circ\varphi_A=\varphi_A$ and so $\xi^{-1}$ is a morphism $(B,\varphi_B)\to(A,\varphi_A)$. Thus $\xi$ is an isomorphism in $\mathcal{C}_\gg^q$.
\end{proof}
	
	Hence $\psi_A$ is actually an isomorphism in $\mathcal{C}_\gg^q$. If we can show that $\psi:=(\psi_A)_{(A,\varphi_A)\in \mathcal{C}_\gg^q}$ is a natural transformation between $(-)^+\otimes \CC_q[U]$ and $\text{id}_{\mathcal{C}_\gg^q}$, then we will have shown that it is a natural isomorphism.
	Let $(A,\varphi_A)$ and $(B,\varphi_B)$ be objects of $\mathcal{C}_\gg^q$ and $\xi:A\to B$ a morphism. Then 
		\begin{align*}
			\psi_B\circ(\xi|_{A^+}\otimes \text{id})&=m_B\circ (\iota_B\otimes \varphi_B)\circ(\xi|_{A^+}\otimes \text{id})\\
				&=m_B\circ(\xi|_{A^+}\otimes \varphi_B)\\
				&=m_B\circ (\xi|_{A^+}\otimes (\xi\circ\varphi_A))\\
				&=m_B\circ (\xi\otimes\xi)\circ(\iota_A\otimes \varphi_A)\\
				&=\xi\circ m_A\circ (\iota_A\otimes \varphi_A)\\
				&=\xi\circ \psi_A
		\end{align*}
	Hence $\psi:(-)^+\otimes\CC_q[U]\Rightarrow \text{id}_{\mathcal{C}_\gg^q}$ is a natural transformation and therefore a natural isomorphism.
	
	Now for every $U_q(\gg^*)$-module $A$, let $\eta_A=\text{id}\otimes1:A\to (A\otimes \CC_q[U])^+$. Then $\eta_A$ is obviously an injective homomorphism of algebras. We need to show that $\eta_A$ is a homomorphism of $U_q(\gg^*)$-module algebras, namely that $\eta_A$ respects the action of $U_q(\gg^*)$. So we make the following computations.
$$		K_i^{\pm1}(\eta_A(a))=K_i^{\pm1}(a\otimes 1)=K_i^{\pm1}(a)\otimes 1
=\eta_A(K_i^{\pm1}(a)),$$
	\begin{align*}
		F_{i,1}(\eta_A(a))&=F_{i,1}(a\otimes 1)\\
			&=F_i(a\otimes 1)-\frac{(1\otimes x_i)(a\otimes 1)-K_i^{-1}(a\otimes 1)(1\otimes x_i)}{q_i-q_i^{-1}}\\
			&=\left(F_{i,1}(a)+F_{i,2}K_i(a)\right)\otimes1+\frac{K_i(a)-K_i^{-1}(a)}{q_i-q_i^{-1}}\otimes x_i-\frac{K_i(a)\otimes x_i+F_{i,2}K_i(a)\otimes 1-K_i^{-1}(a)\otimes x_i}{q_i-q_i^{-1}}\\
			&=F_{i,1}(a)\otimes 1\\
			&=\eta_A(F_{i,1}(a))
	\end{align*}
	\begin{align*}
		F_{i,2}(\eta_A(a))&=F_{i,2}(a\otimes 1)\\
			&=\frac{(1\otimes x_i)K_i^{-1}(a\otimes 1)-(a\otimes 1)(1\otimes x_i)}{q_i-q_i^{-1}}\\
			&=\frac{(1\otimes x_i)(K_i^{-1}(a)\otimes 1)-a\otimes x_i}{q_i-q_i^{-1}}\\
			&=\frac{K_iK_i^{-1}(a)\otimes x_i+(q_i-q_i^{-1})F_{i,2}K_iK_i^{-1}(a)\otimes 1-a\otimes x_i}{q_i-q_i^{-1}}\\
			&=F_{i,2}(a)\otimes 1\\
			&=\eta_A(F_{i,2}(a)).
	\end{align*}
	
	Hence $\eta_A$ respects the action of $U_q(\gg)$. Our last step is to show that $\eta_A$ is surjective. Given an arbitrary element $\displaystyle \sum_{k=1}^n a_k\otimes x_{{\bf j}_k}\in (A\otimes \CC_q[U])^+$ with ${\bf j}_k<{\bf j}_l$ if $k<l$, we have
$$			\sum_{k=1}^n a_k\otimes x_{{\bf j}_k}=E_{\bf i}^{(top)}\left(\sum_{k=1}^n a_k\otimes x_{{\bf j}_k}\right)=a_n\otimes 1.$$
	Hence $(A\otimes \CC_q[U])^+=A\otimes \CC(q)$, so $\eta_A$ is surjective and therefore an isomorphism. One easily checks that $\eta:=(\eta_A)_{A\in U_q(\gg^*)-{\bf ModAlg}}$ is a natural transformation. Since each $\eta_A$ is an isomorphism, $\eta$ is a natural isomorphism $\eta:\text{id}_{U_q(\gg^*)-{\bf ModAlg}}\Rightarrow(-\otimes \CC_q[U])^+$.
	
	We have now shown that $(-)^+\otimes \CC_q[U]\cong \text{id}_{\mathcal{C}_\gg^q}$ and $(-\otimes \CC_q[U])^+\cong \text{id}_{U_q(\gg^*)-{\bf ModAlg}}$, so $\CC_q[U]\otimes -$ and $(-)^+$ are quasi-inverse equivalences of categories.\qed
	
\subsection{Proof of Proposition \ref{prop:qtensor}}
\label{pf:qtensor}

We know by Theorem \ref{thm:qfactoring} that $A\cong A^+\otimes \CC_q[U]$ and $B\cong B^+\otimes \CC_q[U]$ and Theorem \ref{thm:qequiv} says this is an isomorphism of $U_q(\gg)$-module algebras. We now consider the map
	\[\mu_L:(A\underline{\otimes} B)^+\otimes [\varphi_A(\CC_q[U])\otimes \CC(q)]\to A\underline{\otimes} B\cong (A^+\otimes \CC_q[U])\underline{\otimes}(B^+\otimes \CC_q[U])\]
as in Theorem \ref{thm:qfactoring}. As in the proof of Corollary \ref{cor:qhighest} (Section \ref{pf:qhighest}), $\mu_L((A\underline{\otimes} B)^+\otimes [\varphi_A(\CC_q[U])\otimes \CC(q)])$ is a subalgebra of $A\underline{\otimes} B$. Since $A^+\otimes \CC(q)$, $\CC(q)\otimes B^+$, and $\{\varphi_A(x_i)\otimes 1-1\otimes \varphi_B(x_i)\ |\ i\in I\}$ are all contained in $(A\underline{\otimes} B)^+$ and $\{\varphi_A(x_i)\otimes 1\ |\ i\in I\}\subseteq \varphi_A(\CC_q[U])\otimes \CC(q)$, it follows that $\mu_L((A\underline{\otimes} B)^+\otimes [\varphi_A(\CC_q[U])\otimes \CC(q)])$ contains all of these sets. Hence $\mu_L((A\underline{\otimes} B)^+\otimes [\varphi_A(\CC_q[U])\otimes \CC(q)])$ contains a generating set of $A\underline{\otimes} B$. Being a subalgebra, it follows that $\mu_L((A\underline{\otimes} B)^+\otimes [\varphi_A(\CC_q[U])\otimes \CC(q)])=A\underline{\otimes} B$, i.e. $\mu_L$ is surjective. Hence $\mu_L$ is an isomorphism and by Theorem \ref{thm:qfactoring}, $A\underline{\otimes} B$ is adapted and $\nu_{\bf i}(A\underline{\otimes} B\setminus\{0\})=\nu_{\bf i}(\CC_q[U]\setminus\{0\})\ \forall {\bf i}\in R(\wnot)$. Since $A\underline{\otimes}B$ is a $U_q(\gg)$-weight module algebra and $1\otimes \varphi_B$ and $\varphi_A\otimes 1$ are injections, the proposition follows.\qed

\subsection{Proof of Proposition \ref{prop:qfusion}}
\label{pf:qfusion}

The vector space $A\otimes B$ is naturally viewed as a subspace of $A*B$ (or $A\star B$ if you prefer) via $a\otimes b\mapsto (a\otimes 1)\otimes (b\otimes 1)$. In fact, this subspace is actually a subalgebra since
	\[((a\otimes 1)\otimes (b\otimes 1))((a'\otimes 1)\otimes (b'\otimes 1))=q^{(|a'|,|b|)}(aa'\otimes 1)\otimes (bb'\otimes 1)\]
for weight vectors $a,a'\in A\text{ and }b,b'\in B$ of weight $|a|,|a'|,|b|,$ and $|b'|$, respectively. Hence we may equip $A\otimes B$ with this multiplication. 

By design, the prescribed actions of $K_i$ and $F_{i,1}$ on $A\otimes B$ match those on $(A\otimes \CC(q))\otimes (B\otimes \CC(q))\subset A*B$, while the prescribed actions of $K_i^{\pm1}$ and $F_{i,2}$ match those on $(A\otimes \CC(q))\otimes (B\otimes \CC(q))\subset A\star B$. A straightforward check verifies that $F_{i,1}\rhd(F_{j,2}\rhd(a\otimes b))=F_{j,2}\rhd(F_{i,1}\rhd(a\otimes b))$ for $a\in A,\ b\in B$, and $i,j\in I$, so it follows that the prescribed action of $U_q(\gg^*)$ on $A\otimes B$ is well-defined and compatible with multiplication. \qed

\subsection{Proofs of Theorems \ref{thm:b-} and \ref{thm:g}}
\label{pf:b-,g}

We begin with a theorem analogous to Theorem \ref{thm:braidedtriv}. Actually, it follows from \cite[7.3.3]{montgomery}, but for convenience, we give a self-contained proof. 

\begin{theorem}
\label{thm:triv}
	Let $H$ be a Hopf algebra over a field $\kk$ and suppose $A$ is an $H$-module algebra containing $H$ as a subalgebra. Then giving $A$ the structure of an $H$-module algebra via the adjoint action: \[h\rhd a=h_{(1)}aS(h_{(2)})\quad \text{for } h\in H,\ a\in A,\] the linear map $\tau:=(m\otimes id)\circ (id\otimes \iota\otimes id)\circ (id\otimes \Delta):A\rtimes H\to A\otimes_\kk H$ is an algebra isomorphism with inverse $\tau^{-1}=(m\otimes id)\circ (id\otimes\iota\otimes id)\circ(id\otimes S\otimes id)\circ (id\otimes \Delta)$, where $\iota:H\hookrightarrow A$ is the inclusion.
\end{theorem}

\begin{proof}
We first verify that $\tau$ and $\tau^{-1}$ are truly mutually inverse (and hence that we are justified in using the name $\tau^{-1}$). For $a\in A$ and $h\in H$, we directly compute
$$		(\tau\circ\tau^{-1})(a\otimes h)=\tau(aS(h_{(1)})\otimes h_{(2)})
=aS(h_{(1)})h_{(2)}\otimes h_{(3)}
=a\varepsilon(h_{(1)})\otimes h_{(2)}
=a\otimes\varepsilon(h_{(1)})h_{(2)}=a\otimes h,$$
$$		(\tau^{-1}\circ\tau)(a\otimes h)=\tau^{-1}(ah_{(1)}\otimes h_{(2)})
=ah_{(1)}S(h_{(2)})\otimes h_{(3)}
=a\varepsilon(h_{(1)})\otimes h_{(2)}
=a\otimes\varepsilon(h_{(1)})h_{(2)}
=a\otimes h.$$
Since $\tau\circ \tau^{-1}$ and $\tau^{-1}\circ \tau$ act as the identity on pure tensors, they are both the identity homomorphism. Hence $\tau$ and $\tau^{-1}$ are mutually inverse. We conclude by verifying that $\tau$ is actually a homomorphism of algebras (and then $\tau^{-1}$ automatically is as well).
	\begin{align*}
		\tau(a\otimes h)\tau(a'\otimes h')&=(ah_{(1)}\otimes h_{(2)})(a'h'_{(1)}\otimes h'_{(2)})\\
								&=ah_{(1)}a'h'_{(1)}\otimes h_{(2)}h'_{(2)}\\
								&=ah_{(1)}a'h'_{(1)}\otimes \varepsilon(h_{(2)})h_{(3)}h'_{(2)}\\
								&=ah_{(1)}a'\varepsilon(h_{(2)})h'_{(1)}\otimes h_{(3)}h'_{(2)}\\
								&=ah_{(1)}a'S(h_{(2)})h_{(3)}h'_{(1)}\otimes h_{(4)}h'_{(2)}\\
								&=a(h_{(1)}\rhd a')h_{(2)}h'_{(1)}\otimes h_{(3)}h'_{(2)}\\
								&=\tau(a(h_{(1)}\rhd a')\otimes h_{(2)}h')\\
								&=\tau((a\otimes h)(a'\otimes h'))
	\end{align*}
	
	Again, since pure tensors span $A\otimes_\kk H$ and $\tau$ is a linear map, it follows that $\tau$ respects multiplication. The theorem is proved.
\end{proof}

We now apply Theorem \ref{thm:triv} to the situation when $\kk=\CC(q)$ and $H=A=U_q(\bb_-)$, yielding a trivializing isomorphism $\tau:U_q(\bb_-)\rtimes U_q(\bb_-)\to U_q(\bb_-)\otimes U_q(\bb_-)$. Now we also have an embedding of $U_q(\bb_-)$-module algebras $\CC_q[U]\hookrightarrow U_q(\bb_-)$, $x_i\mapsto (q_i-q_i^{-1})F_i$. This induces an embedding of algebras \linebreak$\iota:\CC_q[U]\rtimes U_q(\bb_-)\hookrightarrow U_q(\bb_-)\rtimes U_q(\bb_-)$. Then applying $\tau$, we have an embedding of algebras \linebreak$\tau\circ\iota:\CC_q[U]\rtimes U_q(\bb_-)\hookrightarrow U_q(\bb_-)\otimes U_q(\bb_-)$. Then we have
	\begin{align*}
		(\tau\circ \iota)\left(1\otimes F_i-x_i\otimes \frac{1-K_i^{-2}}{q_i-q_i^{-1}}\right)=&\tau(1\otimes F_i-F_i\otimes (1-K_i^{-2}))\\
		=&F_i\otimes 1+K_i^{-1}\otimes F_i-F_i\otimes 1+F_iK_i^{-2}\otimes K_i^{-2}\\
		=&K_i^{-1}\otimes F_i+F_iK_i^{-2}\otimes K_i^{-2}.
	\end{align*}
Now the families $\{K_i^{-1}\otimes F_i\}_{i\in I}$ and $\{F_iK_i^{-2}\otimes K_i^{-2}\}_{i\in I}$ clearly satisfy the quantum Serre relations  and 
	\begin{align*}
		(F_jK_j^{-2}\otimes K_j^{-2})(K_i^{-1}\otimes F_i)&=F_jK_j^{-2}K_i^{-1}\otimes K_j^{-2}F_i\\
			&=q_i^{c_{i,j}}K_i^{-1}F_jK_j^{-2}\otimes F_iK_j^{-2}\\
			&=q_i^{c_{i,j}}(K_i^{-1}\otimes F_i)(F_jK_j^{-2}\otimes K_j^{-2})
	\end{align*}
for $i,j\in I$. Then by Lemma \ref{lem:qsr}, both families $\{K_i^{-1}\otimes F_i+F_iK_i^{-2}\otimes K_i^{-2}\}_{i\in I}$ and $\left\{1\otimes F_i-x_i\otimes \frac{1-K_i^{-2}}{q_i-q_i^{-1}}\right\}_{i\in I}$ must also satisfy the quantum Serre relations. Since
	\[(1\otimes K_i)\left(1\otimes F_i-x_i\otimes \frac{1-K_i^{-2}}{q_i-q_i^{-1}}\right)(1\otimes K_i^{-1})=q_i^{-c_{i,j}}\left(1\otimes F_i-x_i\otimes \frac{1-K_i^{-2}}{q_i-q_i^{-1}}\right),\]
	 there is a well-defined homomorphism of algebras $U_q(\bb_-)\to \CC_q[U]\rtimes U_q(\bb_-)$ such that ${K_i^{\pm1}}\mapsto1\otimes K_i^{\pm1}$, ${F_i}\mapsto1\otimes F_i-x_i\otimes \frac{1-K_i^{-2}}{q_i-q_i^{-1}}$. 
		
We now wish to dequantize the above map. Set $\mathscr{A}=\{g\in \CC(q)\ |\ g\text{ is regular at }q=1\}$, a local subring of $\CC(q)$ with maximal ideal $(q-1)\mathscr{A}$. Denote by $U_q(\bb_-)_{\mathscr{A}}$ the $\mathscr{A}$-subalgebra of $U_q(\bb_-)$ generated by $\left\{F_i,\frac{K_i^{\pm1}-1}{q_i-1}\ |\ \forall i\in I\right\}$. In fact, it is a Hopf subalgebra over $\mathscr{A}$. Then $(\mathscr{A}\mathcal{B}^{dual})\rtimes U_q(\bb_-)_\mathscr{A}$ is an $\mathscr{A}$-subalgebra of $\CC_q[U]\rtimes U_q(\bb_-)$. Hence the above map restricts to a homomorphism of $\mathscr{A}$-algebras $U_q(\bb_-)_\mathscr{A}\to  (\mathscr{A}\mathcal{B}^{dual})\rtimes U_q(\bb_-)_\mathscr{A}$, inducing a homomorphism of $\CC$-algebras 
\[U(\bb_-)=U_q(\bb_-)_\mathscr{A}/(q-1)U_q(\bb_-)_\mathscr{A}\to [(\mathscr{A}\mathcal{B}^{dual})\rtimes U_q(\bb_-)_\mathscr{A}]/(q-1)[(\mathscr{A}\mathcal{B}^{dual})\rtimes U_q(\bb_-)_\mathscr{A}]=\CC[U]\rtimes U(\bb_-).\]
Hence the following proposition is proven.

\begin{proposition}
 \label{prop:hat}
 The assignments $h_i\mapsto 1\otimes h_i$ and $f_i\mapsto 1\otimes f_i-x_i\otimes h_i$ define a homomorphism of algebras $\hat{\null}:U(\bb_-)\to \CC[U]\rtimes U(\bb_-)$.
\end{proposition} 

Another, less interesting proof of Proposition \ref{prop:hat} involves showing by induction that the elements \linebreak$\hat{f}_i=1\otimes f_i-x_i\otimes h_i\in \CC[U]\rtimes U(\bb_-)$ satisfy

\begin{align*}
 (\text{ad }\hat{f}_i)^{(n)}(\hat{f}_j)=\sum_{k=0}^n\left(\prod_{\ell=n-k}^{n-1}(c_{i,j}+\ell)\right)&\left[x_i^{(k)}\otimes (\text{ad }f_i)^{(n-k)}(f_j)-\frac{d_i}{d_j}\,x_i^{(k-1)}f_i^{(n-k)}(x_j)\otimes f_i\right.\\
\nonumber &\hspace{0.25cm}\left.-\frac{1}{d_j}\,x_i^{(k)}f_i^{(n-k)}(x_j)\otimes ((n-k)d_ih_i+d_jh_j)\right]
 \end{align*}
 
\noindent where we use the conventions $(\text{ad }x)(y)=xy-yx$ and $x_i^{(-1)}=0$. Setting $n=1-c_{i,j}$ and using Remark \ref{rem:nilpotentf}, we verify that 
$(\text{ad }\hat{f}_i)^{(1-c_{i,j})}(\hat{f}_j)=0$.

In any case, we are now ready to prove Theorem \ref{thm:b-}.

\noindent{\bf Proof of Theorem \ref{thm:b-}}.
Since $A$ contains $\CC[U]$ as a $\gg$-module subalgebra, there is an action of $\CC[U]\rtimes U(\bb_-)$ on $A$ given by $(x\otimes g)\blacktriangleright a=xg(a)$. Then $\hat{\null}:U(\bb_-)\to \CC[U]\rtimes U(\bb_-)$ gives rise to an action of $U(\bb_-)$ on $A$ via $g\rhd a=\hat{g}\blacktriangleright a$. Under this action, we have
	\[h_i\rhd a=(1\otimes h_i)\blacktriangleright a=h_i(a)\quad\text{and}\quad f_i\rhd a=(1\otimes f_i-x_i\otimes h_i)\blacktriangleright a=f_i(a)-x_ih_i(a)\]
as prescribed. That $A^+$ is invariant under this action is an easy computation that we will not produce here. \qed

Theorem \ref{thm:b-} of course gives rise to a ``new" $\bb_-$-module algebra structure on $\CC[U]$, satisfying
$$		h_i(x_j)=-c_{i,j}x_j,~h_i(\{x,y\})=\{h_i(x),y\}+\{x,h_i(y)\},~f_i(x)=\frac{1}{2d_i}\{x_i,x\}-\frac{1}{2}x_i h_i(x)$$
for $i\in I$ and $x,y\in \CC[U]$. To distinguish $\CC[U]$ equipped with this action from that with the usual action, we will use $\CC[U]^{op}$ to denote the algebra $\CC[U]$ equipped with the ``new" action. As with $\CC[U]$, we can form the cross product $\CC[U]^{op}\rtimes U(\bb_-)$.

We now consider the well-known homomorphism of algebras $s:\CC[U]\to \CC[U]$ defined by the assignments $s(x_i)=-x_i$ for $i\in I$ and $s(\{x,y\})=-\{s(x),s(y)\}$ for $x,y\in \CC[U]$. We will examine the linear map $s\otimes id: \CC[U]\rtimes U(\bb_-)\to \CC[U]^{op}\rtimes U(\bb_-)$, but first we need the following lemma.

\begin{lemma}
\label{lem:maps}
Given a field $\kk$, let $H$ be a $\kk$-bialgebra generated as a $\kk$-algebra by the subset $X$, $A$ an $H$-module algebra, and $B$ any $\kk$-algebra with multiplication $\mu_B:B\otimes_\kk B\to B$. Given homomorphisms of $\kk$-algebras $\varphi_1:A\to B$ and $\varphi_2:H\to B$, the $\kk$-linear map $\mu_B\circ (\varphi_1\otimes \varphi_2):A\rtimes H\to B$ is a homomorphism of algebras if and only if 
	\begin{equation}
	\label{eq:commute}
		\varphi_2(h)\varphi_1(a)=\varphi_1(h_{(1)}(a))\varphi_2(h_{(2)})
	\end{equation}
for all $h\in X$ and $a\in A$, where we use sumless Sweedler notation: $\Delta(h)=h_{(1)}\otimes h_{(2)}$.
\end{lemma}

\begin{proof}
To simplify notation, we write $\phi:=\mu_B\circ(\varphi_1\otimes \varphi_2)$.

	$(\Rightarrow)$ If $\phi$ is a homomorphism of algebras, then for $h\in X$ and $a\in A$, we have
$$			\varphi_2(h)\varphi_1(a)=\phi(1\otimes h)\phi(a\otimes 1)
=\phi((1\otimes h)(a\otimes 1))
=\phi(h_{(1)}(a)\otimes h_{(2)})
=\varphi_1(h_{(1)}(a))\varphi_2(h_{(2)}).$$

	$(\Leftarrow)$ Let $Y$ be the subset of $H$ consisting of elements $h$ such that \eqref{eq:commute} holds for all $a\in A$. By assumption, $X\subset Y$, so showing that $Y$ is a subalgebra of $H$ is equivalent to showing that $Y=H$. We compute for $h,h'\in Y$ and $a\in A$:
		\begin{align*}
			\varphi_2(h+h')\varphi_1(a)&=(\varphi_2(h)+\varphi_2(h'))\varphi_1(a)\\
			&=\varphi_2(h)\varphi_1(a)+\varphi_2(h')\varphi_1(a)\\
			&=\varphi_1(h_{(1)}(a))\varphi_2(h_{(2)})+\varphi_1(h'_{(1)}(a))\varphi_2(h'_{(2)})\\
			&=\varphi_1((h+h')_{(1)}(a))\varphi_2((h+h')_{(2)})
		\end{align*}
		\begin{align*}
			\varphi_2(hh')\varphi_1(a)&=\varphi_2(h)\varphi_2(h')\varphi_1(a)\\
			&=\varphi_2(h)\varphi_1(h'_{(1)}(a))\varphi_2(h'_{(2)})\\
			&=\varphi_1(h_{(1)}h'_{(1)}(a))\varphi_2(h_{(2)})\varphi_2(h'_{(2)})\\
			&=\varphi_1(h_{(1)}h'_{(1)}(a))\varphi_2(h_{(2)}h'_{(2)})\\
			&=\varphi_1((hh')_{(1)}(a))\varphi_2((hh')_{(2)}).
		\end{align*}
	So we see that $Y$ is closed under addition and multiplication. Since it obviously contains $1$, $Y$ is a subalgebra of $H$ and hence $Y=H$. Thus \eqref{eq:commute} holds for all $h\in H$ and $a\in A$. Now we compute for $h,h'\in H$ and $a,a'\in A$:
	\begin{align*}
		\phi(a\otimes h)\phi(a'\otimes h')&=\varphi_1(a)\varphi_2(h)\varphi_1(a')\varphi_2(h')\\
		&=\varphi_1(a)\varphi_1(h_{(1)}(a'))\varphi_2(h_{(2)})\varphi_2(h')\\
		&=\varphi_1(ah_{(1)}(a'))\varphi_2(h_{(2)}h')\\
		&=\phi(ah_{(1)}(a')\otimes h_{(2)}h')\\
		&=\phi((a\otimes h)(a'\otimes h')).
	\end{align*}
	Since we already knew that $\phi$ was a $\kk$-linear map, it follows that $\phi$ is a homomorphism of $\kk$-algebras.
\end{proof}

\begin{proposition}
	The linear map $s\otimes id:\CC[U]\rtimes U(\bb_-)\to \CC[U]^{op}\rtimes U(\bb_-)$ is a homomorphism of algebras.
\end{proposition}

\begin{proof}
By Lemma \ref{lem:maps}, it suffices to show that $(1\otimes h_i)(s(x)\otimes 1)=s(h_i(x))\otimes 1+s(x)\otimes h_i$ and $(1\otimes f_i)(s(x)\otimes 1)=s(f_i(x))\otimes 1+s(x)\otimes f_i$ for $i\in I$ and $x\in \CC[U]$. It is clear that $s$ respects the action of $h_i$, namely $h_i(s(x))=s(h_i(x))$ for $i\in I$ and $x\in \CC[U]$, so 
$$			(1\otimes h_i)(s(x)\otimes 1)=h_i(s(x))\otimes 1+s(x)\otimes h_i=s(h_i(x))\otimes 1+s(x)\otimes h_i,$$
		\begin{align*}
			(1\otimes f_i)(s(x)\otimes 1)&=f_i(s(x))\otimes 1+s(x)\otimes f_i\\
				&=\left(\frac{1}{2d_i}\{x_i,s(x)\}-\frac{1}{2}x_ih_i(s(x))\right)\otimes 1+s(x)\otimes f_i\\
				&=\left(-\frac{1}{2d_i}\{s(x_i),s(x)\}+\frac{1}{2}s(x_i)s(h_i(x))\right)\otimes 1+s(x)\otimes f_i\\
				&=\left(\frac{1}{2d_i}s(\{x_i,x\})+\frac{1}{2}s(x_ih_i(x))\right)\otimes 1 + s(x)\otimes f_i\\
				&=s\left(\frac{1}{2d_i}\{x_i,x\}+\frac{1}{2}x_ih_i(x)\right)\otimes 1+s(x)\otimes f_i\\
				&=s(f_i(x))\otimes 1+s(x)\otimes f_i.
		\end{align*}
	Hence $s\otimes \text{id}$ is a homomorphism of algebras.
\end{proof}

We are now ready to prove Theorem \ref{thm:g}.

\noindent{\bf Proof of Theorem \ref{thm:g}}.
We observe that $A\otimes \CC[U]^{op}$ is naturally a $\bb_-$-module algebra and therefore is also a $\CC[U]^{op}\rtimes U(\bb_-)$-module. 

\begin{lemma}
\label{lem:b-embed}
There is a homomorphism of algebras $U(\bb_-)\to \CC[U]^{op}\rtimes U(\bb_-)$ given on generators by \[h_i\mapsto 1\otimes h_i\quad \text{and}\quad f_i\mapsto 1\otimes f_i+x_i\otimes h_i.\]
\end{lemma}

\begin{proof} We already saw that there is a homomorphism of algebras $\hat{\null}:U(\bb_-)\to \CC[U]\rtimes U(\bb_-)$ given on generators by $\hat{h}_i= 1\otimes h_i$ and $\hat{f}_i=1\otimes f_i-x_i\otimes h_i$.
Then the composition $(s\otimes \text{id})\circ \hat{\null}:U(\bb_-)\to \CC[U]^{op}\rtimes U(\bb_-)$ has
$		(s\otimes \text{id})(\hat{h}_i)=(s\otimes \text{id})(1\otimes h_i)=1\otimes h_i,~(s\otimes \text{id})(\hat{f}_i)=(s\otimes \text{id})(1\otimes f_i-x_i\otimes h_i)=1\otimes f_i+x_i\otimes h_i$
	as desired.
\end{proof}

The homomorphism in Lemma \ref{lem:b-embed} induces a different action of $\bb_-$ on $A\otimes \CC[U]^{op}$, namely 
\begin{align*}
	h_i\rhd(a\otimes x)&=h_i(a)\otimes x+a\otimes h_i(x)\\
	f_i\rhd(a\otimes x)&=f_i(a)\otimes x+a\otimes f_i(x)+h_i(a)\otimes x_ix+a\otimes x_ih_i(x)\\
	&=f_i(a)\otimes x+h_i(a)\otimes x_ix+a\otimes \left(\frac{1}{2d_i}\{x_i,x\}-\frac{1}{2}x_ih_i(x)+x_ih_i(x)\right)\\
	&=f_i(a)\otimes x+h_i(a)\otimes x_ix+a\otimes \left(\frac{1}{2d_i}\{x_i,x\}+\frac{1}{2}x_ih_i(x)\right)
\end{align*}

Hence there is a well-defined action of $\bb_-$ on $A\otimes \CC[U]$ (note the lack of $op$) given by
\[h_i\rhd(a\otimes x)=h_i(a)\otimes x+a\otimes h_i(x),\quad f_i\rhd(a\otimes x)=f_i(a)\otimes x+h_i(a)\otimes x_ix+a\otimes f_i(x)\]
as prescribed in the theorem. It is clear by definition that the family of operators $\{1\otimes e_i\}_{i\in I}$ which define the action of $e_i$ satisfy the Serre relations. It is also easily checked that 
	\begin{align*}
		h_i\rhd (e_j\rhd (a\otimes x))-e_j\rhd (h_i\rhd (a\otimes x))&=c_{i,j}e_j\rhd(a\otimes x)\\
		h_i\rhd (f_j\rhd (a\otimes x))-f_j\rhd (h_i\rhd (a\otimes x))&=-c_{i,j}f_j\rhd(a\otimes x)\\
		e_i\rhd (f_j\rhd (a\otimes x))-f_j\rhd(e_i\rhd (a\otimes x))&=\delta_{i,j}h_i\rhd(a\otimes x).
	\end{align*}
	
Hence we have a well-defined action of $\gg$ on $A\otimes \CC[U]$. It is clear that each $e_i$ and $f_i$ acts by derivations and $e_i(1\otimes 1)=f_i(1\otimes 1)=0$, so the theorem is proved.\qed

\subsection{Proof of Proposition \ref{prop:tensor}}
\label{pf:tensor}

We simply show that if $(A,\varphi_A)$ and $(B,\varphi_B)$ are objects of $\mathcal{C}_\gg$, then $A\otimes B$ is adapted with $\nu_{\bf i}(A\otimes B\setminus\{0\})=\nu_{\bf i}(\CC[U]\setminus\{0\})$ for all ${\bf i}\in R(\wnot)$. By Theorem \ref{thm:factoring}, it suffices to show that the map $\mu:(A\otimes B)^+\otimes (\varphi_A(\CC[U])\otimes \CC(q))\to A\otimes B$ is surjective. Again by Theorem \ref{thm:factoring}, we know that $A\cong A^+\otimes \CC[U]$ and $B\cong B^+\otimes \CC[U]$. Theorem \ref{thm:equiv} says that these are isomorphisms of $\gg$-module algebras. We observe that the following elements are contained in the image of $\mu$: $a\otimes 1$ for $a\in A^+$, $1\otimes b$ for $b\in B^+$, $\varphi_A(x_i)\otimes 1$ for $i\in I$, and $\varphi_A(x_i)\otimes 1-1\otimes \varphi_B(x_i)$ for $i\in I$. In fact, the image of $\mu$ is a $\gg$-module subalgebra of $A\otimes B$.
\begin{lemma}
	Let $C$ be a $\gg$-module algebra containing $\CC[U]$ as a $\gg$-module subalgebra and let \linebreak$\mu:C^+\otimes \CC[U]\to C$ be restriction of multiplication to these subalgebras. Then the image of $\mu$ is a $\gg$-module subalgebra of $C$. 
\end{lemma} 
\begin{proof}
	Since $\mu$ is $\CC$-linear, it suffices to check that the subspace spanned by the images of pure tensors $c\otimes x$ is closed under multiplication and the $\gg$-action. Now since $\mu(c\otimes x)=cx$, we check only on elements of this form:
$$		(cx)(c'x')=(cc')(xx')=\mu(cc'\otimes xx')\in \mu(C^+\otimes \CC[U]), ~
		e_i(cx)=ce_i(x)=\mu(c\otimes e_i(x))\in \mu(C^+\otimes \CC[U])$$
$$		f_i(cx)=(f_i(c)-h_i(c)x_i)x+h_i(c)x_ix+cf_i(x)=
\mu((f_i(c)-h_i(c)x_i)\otimes x+h_i(c)\otimes x_ix+c\otimes f_i(x))\in \mu(C^+\otimes \CC[U]).$$
	The lemma is proved.
\end{proof}

Since the image of $\mu$ contains a generating set for $A\otimes B$ as a $\gg$-module algebra (see Remark \ref{rem:nilpotentf}), we conclude that $\mu$ is surjective. Hence $\mu$ is an isomorphism and so, by Theorem \ref{thm:factoring}, $A\otimes B$ is adapted with $\nu_{\bf i}(A\otimes B\setminus\{0\})=\nu_{\bf i}(\CC[U]\setminus\{0\})$ for all ${\bf i}\in R(\wnot)$.\qed

\end{document}